\documentclass[a4paper]{amsart}
\usepackage[utf8]{inputenc}
\usepackage[T1]{fontenc}
\usepackage[all]{xy}
\usepackage{hyperref}

\usepackage{amsmath,amstext,amsthm,amssymb,amscd,version, mathrsfs}
\usepackage{xcolor}
\usepackage{comment}
\usepackage{tikz}
\usepackage{tikz-cd}

\def\bA{\mathbb{A}}

\def\bC{\mathbb{C}}

\def\cI{\mathcal{I}}

\def\cO{\mathcal{O}}
\def\cP{\mathcal{P}}

\def\tD{\texttt{D}}
\def\tL{\texttt{L}}
\def\tX{\texttt{X}}
\def\tZ{\texttt{Z}}

\newtheorem{thmintro}{Theorem}

\newtheorem{conjintro}{Conjecture}

\newtheorem{thm}{Theorem}[section]
\newtheorem*{thm*}{Theorem}

\newtheorem{lem}[thm]{Lemma}
\newtheorem{prop}[thm]{Proposition}
\newtheorem{cor}[thm]{Corollary}

\newtheorem{defn}[thm]{Definition}

\newtheorem*{claim}{Claim}

\DeclareMathOperator{\Sp}{Sp}

\DeclareMathOperator{\Hom}{Hom}

\DeclareMathOperator{\Alb}{Alb}

\DeclareMathOperator{\HH}{H}
\DeclareMathOperator{\Spec}{Spec}
\DeclareMathOperator{\Lie}{Lie}
\DeclareMathOperator{\codim}{codim}

\newcommand{\newpar}[1]{\subsection{\texorpdfstring{}{}}}
\newcommand{\parref}[1]{\hyperref[#1]{\S\ref*{#1}}}
\newcommand{\chapref}[1]{\hyperref[#1]{Chapter~\ref*{#1}}}

\counterwithin{equation}{subsection}
\counterwithin{figure}{subsection}

\newcounter{intro}

\newtheorem{intro-conjecture}[intro]{Conjecture}
\newtheorem{intro-corollary}[intro]{Corollary}
\newtheorem{intro-theorem}[intro]{Theorem}

\usepackage{enumitem}

\title[The relative Green-Griffiths-Lang conjecture]{The relative Green-Griffiths-Lang conjecture for families of varieties of maximal Albanese dimension}
\author{Yohan Brunebarbe}
\date{\today}

\begin{document}

\maketitle

\begin{abstract}
We propose a generalization of the Green-Griffiths-Lang conjecture to the relative setting and prove that a strong form of it holds for families of varieties of maximal Albanese dimension.
A key step of the proof consists in a truncated second main theorem type estimate in Nevanlinna theory for families of abelian varieties.
\end{abstract}

\section{Introduction}

\subsection{Morphisms of general type}
Let $k$ be an algebraically closed field of characteristic zero. Recall that a positive-dimensional irreducible smooth proper $k$-variety $X$ is said of general type if its canonical bundle $\omega_{X/k}$ is big, meaning that there exists a positive real number $C$ such that for $m \gg 1$
\[ \dim_k \left( \HH^0(X, (\omega_{X / k})^{\otimes m}) \right) \geq C \cdot m^{\dim X}. \]
More generally, an irreducible proper $k$-variety is said of general type if it is birational to an irreducible smooth proper $k$-variety of general type, and a proper $k$-variety is of general type if at least one of its irreducible components is of general type. By convention, any zero-dimensional $k$-variety is of general type. We propose the following definition of morphisms of general type.

\begin{defn}
A surjective proper morphism $f\colon X \to Y$ between two irreducible $k$-varieties is of general type if, for every geometric point $\bar \eta \colon \Spec \left( \Omega \right) \to Y$ lying over the generic point of $Y$, the proper $\Omega$-variety $X_ {\bar \eta}$ is of general type. More generally, a proper morphism $f\colon X \to Y$ between two $k$-varieties is of general type if at least one of the irreducible components of $X$ is of general type over its image.
\end{defn}

The basic properties of this notion are explored in section \ref{Generalities on proper morphisms of general type}.

\subsection{Special subsets}

For any (non necessarily surjective) proper morphism $X\to Y$ between two complex schemes, the \emph{special} subsets $\Sp_{alg}(X/Y)$, $\Sp_{ab}(X/Y)$, and $\Sp_h(X/Y)$ of $X$ are by definition the union respectively of
\begin{itemize}
\item all (positive-dimensional) irreducible closed subvarieties of $X$ not of general type that are contained in a fibre of $X \to Y$;
\item the images of all non-constant rational maps $A \dasharrow X$ with source an abelian variety $A$ and whose image are contained in a fibre of $X \to Y$;
\item all the entire curves of $X$ contained in a fibre of $X \to Y$.
\end{itemize}

It is not clear from their definition whether these subsets are Zariski-closed in $X$. With the notation from \cite{Bruni-Hyperlargelocalsystem}, for every $\ast \in \{alg, ab, h \}$, one has by definition:
\[\Sp_\ast(X/Y) := \bigcup_{y \in Y} \Sp_\ast(X_y). \]
In particular, by \cite[Proposition 2.1]{Bruni-Hyperlargelocalsystem}, the inclusions 
\begin{equation}\label{ab versus alg}
\Sp_{ab}(X/Y) \subset \Sp_{alg}(X/Y)
\end{equation}
and 
\begin{equation}\label{ab versus h}
\Sp_{ab}(X/Y) \subset \Sp_h(X/Y)
\end{equation}
always hold.

\subsection{The relative Green-Griffiths-Lang conjecture}

We propose the following generalization of the Green-Griffiths-Lang conjecture.
\begin{conjintro}\label{relative Lang conjecture}
Let $X \to Y$ be a proper morphism between two complex algebraic varieties. Then the following conditions are equivalent:
\begin{enumerate}
\item The morphism $X \to Y$ is of general type;
\item $ \Sp_{alg}(X/Y)$ is not Zariski dense in $X$;
\item $ \Sp_{ab}(X/Y)$ is not Zariski dense in $X$;
\item $ \Sp_{h}(X/Y)$ is not Zariski dense in $X$.
\end{enumerate} 
\end{conjintro}

\subsection{Main results}

\begin{thmintro}\label{relative Lang conjecture in the abelian case}
Let $X \to S$ be a proper morphism between two complex algebraic varieties. Assume that there exists an abelian scheme $A \to S$, an $A$-torsor $P \to S$ and a finite $S$-morphism $X \to P$. Then, 
\begin{enumerate}
\item The inclusions in (\ref{ab versus alg}) and (\ref{ab versus h}) are equalities;\\
In such a case, the special subset is denoted by $\Sp(X/Y)$ without any risk of confusion.
\item $ \Sp(X/Y)$ is Zariski-closed in $X$;
\item $ \Sp(X/Y) \neq X$ if and only if the morphism $X \to Y$ is of general type.
\end{enumerate} 
In particular, the set of $s \in S$ such that $\Sp(X_s) = \varnothing$ (respectively $\Sp(X_s) \neq X_s$) is Zariski open in $S$.  
\end{thmintro}

\begin{thmintro}\label{main corollary}
Let $f\colon X \to Y$ be a smooth proper surjective morphism with connected fibres between complex algebraic varieties. Assume that there exists a geometric point $\bar \eta \colon \Spec \left( \Omega \right) \to Y$ lying over $y_o \in Y$ such that the fibre $X_{\bar \eta}$ is of general type and of maximal Albanese dimension.\\
Then, for every $\ast \in \{alg, ab , h\}$, $\Sp_{\ast}(X/Y)$ is not Zariski-dense in $X$, and for  every closed point $y \in Y$ in a Zariski neighborhood of $y_o$, $\Sp_{\ast}(X_y)$ is not Zariski-dense in $X_y$. 
\end{thmintro}

Recall that a smooth proper complex algebraic variety is said of maximal Albanese dimension if its Albanese morphism is generically finite, or equivalently if it admits a generically finite morphism to an abelian variety. \\

As a key step toward Theorem \ref{relative Lang conjecture in the abelian case}, we prove the following truncated second main theorem type estimate in Nevanlinna theory.

\begin{thmintro}\label{truncated SMT for abelian schemes}
Let $S$ be a complex algebraic variety. Let $A \to S$ be an abelian scheme. Let $X \to S$ be an $A$-torsor over $S$. Let $D \subset X$ be an effective reduced Cartier divisor. Let $L \to X$ be a line bundle which is relatively ample over $S$. For every $\epsilon >0$, there exists $\Xi = \Xi(S, A, X, D, L, \epsilon)  \subsetneq X$ a closed subscheme such that for every $s \in S$, if $f \colon \bC \to X_s$ is the orbit of a one-parameter subgroup of $A_s$ such that $f (\bC) \not \subset \Xi$, then $D_s$ is a reduced effective Cartier divisor and 
\[ T(r,f,D_s) \leq N^{(1)}(r,f, D_s) + \epsilon \cdot T(r,f,L_s) \; ||_{\epsilon} . \]
\end{thmintro}

Here the symbol $||_{\epsilon}$ means that the stated estimate holds for $r > 0$ outside some exceptional Borel subset with finite Lebesgue measure, which depends on $\epsilon$. The definitions of Nevanlinna's characteristic and truncated counting functions $T$ and $N^{(1)}$ are recalled in section \ref{Notations Nevanlinna}.

\subsection{Earlier results.}

When the finite $S$-morphism $X \to P$ is a closed immersion, Theorem \ref{relative Lang conjecture in the abelian case} is a direct consequence of works of Bloch \cite{Bloch}, Ueno \cite{Ueno-book}, Ochiai \cite{Ochiai}, Kawamata \cite{Kawamata-characterization, Kawamata-Bloch}, Abramovich \cite{Abramovich} and McQuillan \cite{McQuillan}.
When $S$ is a point and the finite $S$-morphism $X \to P$ is surjective, the points $(1)$ and $(3)$ of Theorem \ref{relative Lang conjecture in the abelian case} were proved to be true in a fundamental work of Yamanoi \cite{Yamanoi-maximal-Albanese}. This case is substantially more difficult than the case where $X \to P$ is a closed immersion. It seems that the point $(2)$ is new already in this special case. However, our main contribution in this paper is the extension of Yamanoi's work to the relative setting. When $S$ is a point, Yamanoi has proved that Theorem \ref{truncated SMT for abelian schemes} holds more generally for every non-constant holomorphic map $f \colon \bC \to X$. Our proof of Theorem \ref{truncated SMT for abelian schemes} is very much inspired by Yamanoi works. However, the relative setting introduces some significant new difficulties, even when one restrict ourselves to translates of one-parameter subgroups. 

\subsection{Acknowledgements.}
I warmly thank P. Autissier, F. Campana and M. Maculan for interesting discussions.
 

\subsection{Conventions} 
We recall some definitions that we will use freely in the sequel.
\begin{itemize}
\item An algebraic variety over a field $k$ is a reduced separated scheme of finite type over $k$.
\item A geometric point of a $k$-variety $X$ is a $k$-morphism $\bar x \colon \Spec(\Omega) \to X$, where $\Omega \supset k$ is an algebraically closed field. If $x \in |X|$ is the image of $\bar x$, we say that $\bar x$ is a geometric point lying over $x$.
\item A morphism $X \to Y$ of schemes is called dominant if its image is a dense subset of $Y$. 
\item Let $X \to Y$ be a dominant morphism between two $k$-algebraic varieties, with $Y$ integral. A geometric generic fibre of $X \to Y$ is a geometric fibre $X_{\bar\eta} = X \times_Y \Spec(\Omega)$ associated to a geometric point $\bar \eta \colon \Spec(\Omega) \to Y$ lying over the generic point $\eta$ of $Y$.
\item If $X \to S$ and $T \to S$ are two morphisms of schemes, then we will denote by $X_T \to T$ the base-change of $X \to S$ along $T \to S$.
\end{itemize}

\section{Preliminaries}

\subsection{Fibrations}
Throughout this section $k$ is a field of characteristic zero. 

\begin{defn}
A fibration $f \colon X \to Y$  is a proper dominant morphism of $k$-algebraic varieties such that $f_\ast(\cO_X ) = \cO_Y$.
\end{defn}
In particular, a fibration is a surjective morphism with connected fibres. Conversely, if $X$ is normal, then a surjective proper morphism with connected fibres is a fibration. The composition of two fibrations is a fibration. Thanks to the Stein factorization, any proper morphism $f$ can be decomposed as $f = g \circ h$ with $g$ finite and $h$ a fibration.

\begin{prop}\label{fibration-normal and integral}
Let $f \colon X \to Y$ be a fibration between $k$-algebraic varieties. If $X$ is integral (respectively normal), then $Y$ is integral (respectively normal).
\end{prop}

\begin{prop}\label{fibration-generic fibre}
Let $f \colon X \to Y$ be a fibration between normal irreducible $k$-algebraic varieties. Then, the generic fibre of $f$ is geometrically irreducible.
\end{prop}
\begin{proof}
The generic fibre is both geometrically connected and geometrically normal, hence it is geometrically irreducible.
\end{proof}

\subsection{Torsors} 
\begin{defn}
Let $S$ be a scheme, $X$ a $S$-scheme and $G$ a $S$-group scheme that acts on $X$ by means of a morphism $ G \times_S X \to X,  (g,x) \mapsto g \cdot x$.
Then $X$ is a $G$-torsor over $S$ if 
\begin{enumerate}
\item the structural morphism $X \to S$ is faithfully flat and quasi-compact, and
\item the morphism $G \times_S X \to X \times_S X$, $(g,x) \mapsto (g \cdot x, x)$, is an isomorphism.
\end{enumerate}
\end{defn}
If $G$ is smooth, then $X$ can be trivialized by a surjective étale cover $S^\prime \to S$.

\begin{prop}
Let $A$ be an abelian scheme over a locally noetherian basis $S$. Let $B$ be an abelian subscheme of $A$. Let $X$ be an $A$-torsor over $S$. Then the quotient $X \slash B$ exists as an $S$-scheme and the quotient morphism $X \to X \slash B$ is a $B_{X \slash B}$-torsor.
\end{prop}

\subsection{Jet spaces}

For every non-negative integer $m$, we set $\Delta_m := \Spec \left( \bC[t]/(t^{m+1}) \right)$.
\begin{defn}[cf. {\cite{Vojta-jets}}]
Let $X \to Y$ be a morphism of complex schemes. Let $m \geq 0$ be an integer. The functor from $Y$-schemes to sets, given by 
\( Z \mapsto  \Hom_Y \left(Z \times_{\bC} \Delta_m , X \right),  \)
is represented by a $Y$-scheme $J_m(X/Y)$, which is called the scheme of $m$-jet differentials
of $X$ over $Y$.
\end{defn}

We collect some easy consequences of the definition, cf. \cite{Vojta-jets}.

\begin{prop}\label{jets_properties}
Let $X \to Y$ be a morphism of complex schemes.
\begin{enumerate}
\item There are natural 'forget maps' :
\[   J_{m+1}(X/Y) \to J_m(X/Y)   \to \cdots \to J_1(X/Y) \to J_0(X/Y) = X  . \]
For all $m$, $J_m(X/Y)$ is affine (and therefore quasi-compact and separated) over $X$.

\item For every $m$, there is a “zero section” $s_m \colon X \to J_m(X/Y)$
satisfying $\pi_m \circ s_m = \mathrm{Id}_X$ for all $m$ and $\pi_{ji} \circ s_j = s_i$ for all $0 \leq i \leq j$.

\item For every $m$, the formation of $J_m(X/Y )$ is functorial in pairs $X \to Y$, i.e. a commutative diagram 
\[
\begin{tikzcd}
X \arrow[d] \arrow[r, "f"] & X^\prime \arrow[d] \\
Y \arrow[r]& Y^\prime
\end{tikzcd}
\]

induces a commutative diagram

\[
\begin{tikzcd}
J_m(X/Y) \arrow[d] \arrow[r, "J_m(f)"] & J_m(X^\prime/Y^\prime) \arrow[d] \\
X \arrow[r]& X^\prime.
\end{tikzcd}
\]
If the first diagram is cartesian, then the second diagram is cartesian. If $f$ is a closed immersion, then $J_m(f)$ is a closed immersion.

    \item For every morphism of schemes $X_1 \to Y$ and $X_2 \to Y$ and for every $m$, there is a canonical isomorphism $J_m(X_1 \times_Y X_2 / Y ) \simeq J_m(X_1 /Y) \times_Y J_m(X_2 / Y )$ of $Y$-schemes.
   
    \item If $G$ is a $S$-group scheme and $m \geq 0$ an integer, then $J_m(G/S)$ is also a $S$-group scheme. If $G$ acts on a $S$-scheme $X$, then the $S$-group scheme $J_m(G/S)$ acts on the $S$-scheme $J_m(X/S)$. If moreover $X$ is a $G$-torsor, then  $J_m(X/S)$ is a $J_m(G/S)$-torsor.
    
    \item If $f \colon X \to Y$ is a smooth morphism, then $\pi_{ji} \colon J_j (X/Y) \to J_i(X/Y )$ is surjective for every integers $0 \leq i \leq j$.
\end{enumerate}

\end{prop}

\begin{defn}
Let $S$ be a scheme and $G$ a $S$-group scheme. For every positive integer $k$, we define the $S$-group scheme $\Lie^k(G /S)$ by the following exact sequence of $S$-group schemes:
\[ 1 \to \Lie^k(G /S) \to J_k(G / S) \to G \to 1 . \]
By convention, $\Lie(G /S) := \Lie^1(G /S)$. 
\end{defn}

In particular, for every positive integers $k \leq l$, the morphism of $S$-group schemes  $J_l(G / S) \to J_k(G / S)$ induces a morphism of $S$-group schemes $\Lie^l(G /S) \to \Lie^k(G /S)$. \\

The morphism of $S$-group schemes $J_k(G / S) \to G$ has a canonical section (the ``zero section'' from Proposition \ref{jets_properties} (2)). Therefore, when $G$ is commutative, there is a canonical splitting as $S$-group schemes:
\[ J_k(G / S) = G \times_S \Lie^k(G /S)  . \]

By working fppf-locally, we immediately infer the following result.

\begin{prop}
Let $S$ be a scheme, $G$ a commutative $S$-group scheme and $X$ a $G$-torsor. Then, for every positive integer $k$, there is a cartesian square
\[
\begin{tikzcd}
J_k(X / S) \arrow[d] \arrow[r] & \Lie^k(G/S) \arrow[d] \\
X \arrow[r]& S.
\end{tikzcd}
\]
\end{prop}

\subsection{Canonical sections of the forget maps}\label{Canonical sections of the forget maps}
Let $S$ be a complex scheme and $A$ an $S$-abelian scheme.  By composing the exponential morphism $\Lie(A /S)^{an} \to A^{an}$ with the scalar multiplication $\bA^1_S \times_S \Lie(A /S) \to \Lie(A /S)$ for the vector bundle structure on $\Lie(A /S)$, one obtains an analytic map 
\[  \left(\bA^1_S \right)^{an} \times_{S^{an}} \left(\Lie(A /S)\right)^{an} \to A^{an} . \]
For every positive integer $k$, by precomposing with the analytification of the closed immersion $(\Delta_k)_{S} \hookrightarrow \bA^1_S $, one obtains an analytic map 
\[  (\Delta_k)_{S^{an}} \times_{S^{an}} \left(\Lie(A /S) \right)^{an} \to A^{an} . \]
By the universal property of the (analytic version of the) jet spaces, this provides us with a compatible system of holomorphic maps 
\[ s_k \colon \Lie(A /S)^{an} \to \Lie^k(A /S)^{an} \subset J_k(A /S)^{an}.  \]
It follows easily from the construction that $s_k \colon \Lie(A /S)^{an} \to \Lie^k(A /S)^{an}$ is a morphism of analytic vector bundles, whose composition with the canonical `forget map' $ \Lie^k(A /S)^{an} \to \Lie(A /S)^{an} $ equals the identity of $\Lie(A /S)^{an} $. Finally, it follows from GAGA that $\sigma_k$ is induced by a morphism of the underlying algebraic vector bundles $\sigma_k \colon \Lie(A /S) \to \Lie^k(A /S) $. In the sequel, we will refer to the $s_k$'s and $\sigma_k$'s as the \textit{canonical sections} of the forget maps.

\subsection{Relative Cartier divisors}
\begin{defn}
Let $f\colon X \to S$ be a morphism of schemes. A relative effective Cartier divisor on $X/S$ is an effective Cartier divisor $D \subset X$ such that $D \to S$ is a flat morphism of schemes. 
\end{defn}

The base-change of a relative effective Cartier divisor is still a relative effective Cartier divisor, see \cite[\href{https://stacks.math.columbia.edu/tag/056P}{Tag 056P}]{stacks-project}.

\section{Generalities on proper morphisms of general type}\label{Generalities on proper morphisms of general type}

Throughout this section $k$ is an algebraically closed field of characteristic zero.

\begin{prop}\label{base-field extension general type}
 Let $L \supset K$ be an extension of algebraically closed fields of characteristic zero. Let $X$ be a proper $K$-variety. Then the $K$-variety $X$ is of general type if and only if the $L$-variety $X_L$ is of general type.
\end{prop}
\begin{proof}
Assume first that $X$ is irreducible. Let $Y \to X$ be a birational $K$-morphism from a smooth proper irreducible $K$-variety. Then $Y_L$ is a smooth proper irreducible $L$-variety and the $K$-morphism $Y \to X$ induces a birational $L$-morphism $Y_L \to X_L$. Since for every positive integer $m$ 
\[\dim_K \left( \HH^0(Y, \omega^{\otimes m}_{Y / K}) \right)= \dim_L \left( \HH^0(Y_L, \omega^{\otimes m}_{Y_L / L}) \right),\] it follows that the $K$-variety $Y$ is of general type if and only if the $L$-variety $Y_L$ is of general type. This proves the proposition when $X$ is irreducible. \\
In general, if $X = \cup_i X_i$ is the decomposition in irreducible components of $X$, then $X_L= \cup_i (X_i)_L$ is the decomposition in irreducible components of $X_L$. Thanks to the preceding discussion, the $K$-variety $X_i$ is of general type for some $i$ if and only if the $L$-variety $(X_i)_L$ is of general type for some $i$. This proves the proposition in general.
 \end{proof}

\begin{cor}\label{criterion-general type}
Let $f \colon X \to Y$ be a surjective proper $k$-morphism between two irreducible $k$-varieties. Assume that there exists a geometric point $\bar \eta \colon \Spec \left( \Omega \right) \to Y$ lying over the generic point of $Y$ such that the proper $\Omega$-variety  $X_ {\bar \eta}$ is of general type. Then $f$ is of general type. 
\end{cor}
\begin{proof}
Since any two geometric points lying over the generic point of $Y$ are dominated by a third one, it is a consequence of Proposition \ref{base-field extension general type}.
\end{proof}

\begin{prop}\label{base change-general type}
Let $f \colon X \to Y$ be a surjective proper $k$-morphism between two irreducible $k$-varieties. Let $Z$ be an irreducible $k$-variety and $Z \to Y$ be a dominant $k$-morphism. Then $f \colon X \to Y$ is of general type if and only if $f_Z \colon X_Z \to Z$ is of general type.
\end{prop}
\begin{proof}
If $\bar \eta_Z \colon \Spec(\Omega) \to Z$ is a geometric point lying over the generic point of $Z$, then the composite morphism $\bar \eta_Y \colon \Spec(\Omega) \to Z \to Y$ is a geometric point lying over the generic point of $Y$. Therefore, the geometric generic fibre of $X_Z \to Z$ at $\bar \eta_Z$ is isomorphic as $\Omega$-scheme to the geometric generic fibre of $X \to Y$ at $\bar \eta_Y$. The result follows.
\end{proof}

\begin{prop}\label{normalization-general type}
Let $f \colon X \to Y$ be a surjective proper $k$-morphism between irreducible $k$-varieties. Let $\nu \colon \tilde{X} \to X$ denote the normalization of $X$. Then $f$ is of general type if and only if the composite morphism $f \circ \nu$ is of general type.
\end{prop}
\begin{proof}
Let $\bar \eta \colon \Spec(\Omega) \to Y$ be a geometric point lying over the generic point of $Y$. The induced morphism $\tilde{X}_{\bar \eta} \to X_{\bar \eta}$ between the geometric generic fibres at $\bar \eta$ of $f \circ \nu$ and $f$ respectively coincides with the normalization of $X_{\bar \eta}$ \cite[Lemma 2.6]{Brunebarbe-Maculan}. If $X_{\bar \eta} = \cup_i X_i$ is the decomposition in irreducible components of $X_{\bar \eta}$, then $\tilde{X}_{\bar \eta} = \cup_i \tilde{X}_i$ is the decomposition in irreducible components of $\tilde{X}_{\bar \eta}$. The proper $\Omega$-variety $X_{\bar \eta}$ is of general type if and only if $X_i$ is of general type for some $i$, if and only if $\tilde{X}_i$ is of general type for some $i$, if and only if $\tilde{X}_{\bar \eta}$ is of general type.
\end{proof}

\begin{prop}\label{Stein factorization-general type}
Let $X \to Y$ be a surjective proper $k$-morphism between two irreducible $k$-varieties, and let $X \to Z \to Y$ be its Stein factorization. Then $X \to Y$ is of general type if and only if $X \to Z$ is of general type.
\end{prop}
\begin{proof}
Let $\bar \eta \colon \Spec(\Omega) \to Y$ be a geometric point lying over the generic point of $Y$.
Since the induced $\Omega$-morphism $Z_{\bar \eta} \to \Spec\left( \Omega \right)$ is finite surjective and $Z_{\bar \eta} $ is reduced, the $\Omega$-variety $Z_{\bar \eta} $ is a finite disjoint union of copies of $\Spec\left( \Omega \right)$. Therefore, the geometric generic fibre $X_{\bar \eta}$ of the $k$-morphism $X \to Y$ is a finite disjoint union of geometric generic fibres of the $k$-morphism $X \to Z$. The result follows thanks to Corollary \ref{criterion-general type}.
\end{proof}

\begin{prop}\label{finite morphism-general type}
Let $X \to Y$ and $Y \to Z$ be surjective proper $k$-morphisms between irreducible $k$-varieties.
If $X \to Y$ is finite and $Y \to Z$ is of general type, then $X \to Z$ is of general type.
\end{prop}
\begin{proof}
If $\bar \eta \colon \Spec(\Omega) \to Z$ is a geometric point lying over the generic point of $Z$, then $X \to Y$ induces a finite surjective $\Omega$-morphism $X_{\bar \eta} \to Y_{\bar \eta}$. Since $Y \to Z$ is of general type, there exists $Y^\prime$ an irreducible component of $ Y_{\bar \eta}$ which is a $\Omega$-variety of general type. Since $X_{\bar \eta} \to Y_{\bar \eta}$ is finite surjective, there exists $X^\prime$ an irreducible component of $X_{\bar \eta}$ such that $X_{\bar \eta} \to Y_{\bar \eta}$ induces a finite surjective $\Omega$-morphism $X^\prime \to Y^\prime$. Since $X^\prime$ and $Y^\prime$ are integral, it follows that $X^\prime$ is a $\Omega$-variety of general type. A fortiori, $X_{\bar \eta}$ is a $\Omega$-variety of general type. 
\end{proof}

\begin{prop}\label{composition general type}
Let $X \to Y$ and $Y \to Z$ be surjective proper morphisms between irreducible $k$-varieties.
\begin{enumerate}
\item If both $X \to Y$ and $Y \to Z$ are of general type, then the composite morphism $X \to Z$ is of general type.
\item If the composite morphism $X \to Z$ is of general type, then $X \to Y$ is of general type.
\end{enumerate}
\end{prop}

\begin{proof}
\begin{enumerate}
\item Thanks to Proposition \ref{normalization-general type}, one can assume that $X$ is normal. If $X \to Y^\prime \to Y$ is the Stein factorization of $X \to Y$, then $X \to Y^\prime$ is of general type thanks to Proposition \ref{Stein factorization-general type}. Since $Y^\prime \to Y$ is finite surjective, it follows from Proposition \ref{finite morphism-general type} that the composite morphism $Y^\prime \to Y \to Z$ is of general type. Therefore one can assume from the beginning that $X \to Y$ is a fibration. In particular $Y$ is normal.\\
Let $\bar \eta_Y \colon \Spec(\Omega) \to Y$ be a geometric point lying over the generic point of $Y$, so that the composite morphism $\bar \eta_Z \colon \Spec(\Omega) \to Y \to Z$ is a geometric point lying over the generic point of $Z$. The fibration $X \to Y$ induces a fibration $X_{\bar \eta_Z} \to Y_{\bar \eta_Z}$ between two normal proper $\Omega$-varieties. The geometric generic fibre of $X_{\bar \eta_Z} \to Y_{\bar \eta_Z}$ at $\bar \eta_Y$ is isomorphic as $\Omega$-scheme to the geometric generic fibre of $X \to Y$ at $\bar \eta_Y$, therefore the $\Omega$-morphism $X_{\bar \eta_Z} \to Y_{\bar \eta_Z}$ is of general type. Since $Y_{\bar \eta_Z} \to \Spec \left( \Omega \right)$ is of general type by assumption, it follows from a result of Koll\'ar \cite[p.363, Theorem]{Kollar-subadditivity} that $X_{\bar \eta_Z} \to \Spec \left( \Omega \right)$ is of general type.

\item Thanks to Proposition \ref{normalization-general type}, one can assume that $X$ is normal. If $X \to Y^\prime \to Y$ is the Stein factorization of $X \to Y$, then it is sufficient to prove that $X \to Y^\prime$ is of general type thanks to Proposition \ref{Stein factorization-general type}. Therefore, up to replace $X \to Y$ with $X \to Y^\prime$, and $Y \to Z$ with the composite morphism $Y^\prime \to Y \to Z$, one can assume from the beginning that $X \to Y$ is a fibration.\\
Let $\bar \eta_Y \colon \Spec(\Omega) \to Y$ be a geometric point lying over the generic point of $Y$, so that the composite morphism $\bar \eta_Z \colon \Spec(\Omega) \to Y \to Z$ is a geometric point lying over the generic point of $Z$. The fibration $X \to Y$ induces a fibration $X_{\bar \eta_Z} \to Y_{\bar \eta_Z}$ between the corresponding geometric generic fibres. Since by assumption $X_{\bar \eta_Z} \to \Spec \left(\Omega \right)$ is of general type, it follows from Iitaka's easy addition formula \cite[Theorem 11.9]{Iitaka-book} (see also \cite[Lemma 2.3.31]{Fujino-book}) that every geometric generic fibre of $X_{\bar \eta_Z} \to Y_{\bar \eta_Z}$ is of general type.  The geometric generic fibre of $X_{\bar \eta_Z} \to Y_{\bar \eta_Z}$ at $\bar \eta_Y$ is isomorphic as $\Omega$-scheme to the geometric generic fibre of $X \to Y$ at $\bar \eta_Y$, hence $X \to Y$ is of general type. 
\end{enumerate}
\end{proof}

In the preceding proof, we have used some results that are sometime stated in term of a very general fibre instead of the geometric generic fibre. The link is provided by the following result. 

\begin{lem}\label{Kodaira dimension-general versus generic}
Let $f \colon X \to Y$ be a projective surjective morphism between complex varieties. Assume that $Y$ is irreducible and that the generic fibre of $f$ is geometrically irreducible. Then the Kodaira dimension of a very general fibre is equal to the Kodaira dimension of any geometric generic fibre.
\end{lem}
\begin{proof}
This is well-known, at least when in addition $X$ and $Y$ are normal and $f$ has connected fibres, cf. \cite[Lemma 2.3.28]{Fujino-book}. In general, let $\tilde{X} \to X$ be the normalization of $X$. Up to replacing $Y$ with a dense Zariski open subset, one can assume that $Y$ is normal and that for every $y \in Y$ the induced morphism $\tilde{X}_y \to X_y$ is the normalization of $X_y$, cf. \cite[Lemma 2.6]{Brunebarbe-Maculan}. By further shrinking $Y$, one can assume that every fibre of $f$ is geometrically integral. Therefore, for every closed point $y \in Y$, the fibre at $y$ of the composite morphism $\tilde{X} \to Y$ is irreducible, hence connected. The result follows by applying \cite[Lemma 2.3.28]{Fujino-book} to the projective morphism $\tilde{X} \to Y$, and using that by definition an irreducible variety over a field and its normalization have the same Kodaira dimension.
\end{proof}

The next result shows that being of general type is stable by generization.

\begin{thm}\label{generization general type}
Let $X \to Y$ be a surjective projective morphism between irreducible complex varieties. If there exists a geometric point $\bar \eta \colon \Spec \left( \Omega \right) \to Y$ such that $X_{\bar \eta}$\footnote{By convention, we say that a proper scheme of finite type over $\Omega$ is of general type if its reduction is a $\Omega$-variety of general type} is of general type, then $X \to Y$ is of general type.   
\end{thm}
In other words, if a surjective projective morphism $f \colon X \to Y$ is not of general type, then every irreducible component of a fibre of $f$ is not of general type.

\begin{proof}
When $X \to Y$ is a projective fibration from a normal variety to a smooth curve and $\bar \eta$ lies over a closed point of $Y$, this is a result of Nakayama, see \cite[Theorem VI.4.3]{Nakayama}. The general case reduces to this particular case as follows.\\

Let $g \colon \tilde{X} \to Y$ be the composition of $f \colon X \to Y$ with the normalization $\tilde{X} \to X$.
Since the induced morphism $\tilde{X}_{\bar \eta} \to X_{\bar \eta}$ is surjective and finite, the $\Omega$-variety $\tilde{X}_{\bar \eta}$ is of general type. Therefore, thanks to Proposition \ref{normalization-general type}, one can assume from the beginning that $X$ is normal. \\

Let $X \to Z \to Y$ be the Stein factorization of $X \to Y$. Since the $\Omega$-variety $X_{\bar \eta}$ is of general type, and $X_{\bar \eta}$ is the disjoint union of the $\Omega$-varieties $X_{\bar \eta^\prime}$ for all $\bar \eta^\prime \in Z \left( \Omega \right)$ lying over $\bar \eta$, there exists $\bar \eta^\prime \in Z \left( \Omega \right)$ such that $X_{\bar \eta^\prime}$ is of general type. Therefore, in view of Proposition \ref{Stein factorization-general type}, one can assume that $X \to Y$ is a fibration.\\

Thanks to Lemma \ref{Kodaira dimension-general versus generic}, it is sufficient to prove the case where $\bar \eta$ lies over a closed point of $Y$. \\

Therefore, there exists a dense Zariski open subset $U \subset Y$ such that the fibre $X_u$ is (geometrically) normal for every $u \in U$. Thanks to Lemma \ref{Kodaira dimension-general versus generic}, there exists a countable union of closed subvarieties $Z \neq Y$ such that the Kodaira dimension of $X_y$ is constant for every $y \notin Z$. It is harmless to assume that $X \backslash Z \subset U$. By assumption, there exists $y_o \in Y$ a closed point such that $X_{y_o}$ is of general type. Let $C \to Y$ be a complex morphism from a smooth complex curve $C$, whose image contains $y_o$ and is not contained in $Z$. The base-change $X_C \to C$ is a projective fibration. Precomposing with the normalization of $X_C$ (this does not change anything over $C \cap U$), we get a projective fibration to which the result of Nakayama applies. The general result follows.
\end{proof}

\section{Generalities on special sets}

We collect some results on special sets. 
\begin{prop}\label{base change-special sets}
Fix $\ast \in \{h,ab, alg \}$. Let $X \to Y$ be a proper morphism between two complex algebraic varieties. Let $Z$ be a complex algebraic variety equipped with a morphism $Z \to Y$. Then  $\Sp_\ast \left(X_Z \slash Z \right)  $ is the preimage of $\Sp_\ast \left(X \slash Y \right)$ along the canonical map $X_Z \to X$.
\end{prop}

\begin{cor}\label{base change-special sets}
Fix $\ast \in \{h,ab, alg \}$. Let $X \to Y$ be a surjective proper morphism between two complex irreducible algebraic varieties. Let $Z$ be a complex irreducible algebraic variety equipped with a dominant morphism $Z \to Y$. Then, $\Sp_\ast \left(X \slash Y \right)$ is Zariski-dense in $X$ if and only if $\Sp_\ast \left(X_Z \slash Z \right)$ is Zariski-dense in $X_Z$.  
\end{cor}
\begin{proof}
Assume first that the morphism $Z \to Y$ is universally open, so that the induced morphism $X_Z \to X$ is open. Then  $\overline{\Sp_\ast \left(X_Z \slash Z \right)}  $ is the preimage of $\overline{\Sp_\ast \left(X \slash Y \right)}$ along the dominant map $X_Z \to X$, and the result follows.
For the general case, let $U \subset Y$ be a dense open subset such that the induced morphism $Z_U \to U$ is flat, hence universally open. By applying the special case to the dominant universally open morphisms $U \subset Y$, $Z_U \to U$ and $Z_U \subset Z$, one gets that $\Sp_\ast \left(X \slash Y \right)$ is Zariski-dense in $X$ if and only if $\Sp_\ast \left(X_U \slash U \right)$ is Zariski-dense in $X_U$, if and only if  $\Sp_\ast \left(X \times_ U Z_U\slash Z_U \right)$ is Zariski-dense in $X \times_ U Z_U$, if and only if $\Sp_\ast \left(X_Z \slash Z \right)$ is Zariski-dense in $X_Z$.  
\end{proof}

\begin{prop}\label{composition-special sets}
Fix $\ast \in \{h,ab, alg \}$. Let $f\colon X \to Y$ and $g \colon Y \to Z$ be two proper morphisms between complex algebraic varieties. Then $h := g \circ f$ is a proper morphism and $\Sp_\ast(X \slash Z) \subset \Sp_\ast\left(X \slash Y \right) \cup f^{-1}\left(\Sp_\ast\left(Y \slash Z \right) \right)$. If moreover $g$ is finite, then $\Sp_\ast \left(X \slash Z \right) = \Sp_\ast \left(X \slash Y \right)$.
\end{prop}
\begin{proof}
Let $\bC \to X$ be a non-constant holomorphic map whose image is contained in a fibre of the composition $h$. Either the composition $\bC \to X \to Y$ is constant, so that the image of $\bC \to X$ is contained in $\Sp_h(X \slash Y)$, either it is not constant, so that the image of $\bC \to X$ is contained in $f^{-1}(\Sp_h(Y \slash Z))$. The proof for $\ast = ab$ is similar.  Finally, let us prove the case $\ast = alg$. Let $W$ be a positive-dimensional irreducible closed subvariety of $X$ contained in a fibre of $X \to Z$, and assume that $W \not\subset \Sp_{alg}\left(X \slash Y \right) \cup f^{-1}\left(\Sp_{alg}\left(Y \slash Z \right) \right)$. Then $f(W)$ is either a point or a positive-dimensional irreducible closed subvariety of $Y$. In the first case, since $W \not\subset \Sp_{alg}\left(X \slash Y \right)$, $W$ is necessarily of general type. In the second case, since $W \not\subset f^{-1}\left(\Sp_{alg}\left(Y \slash Z \right) \right)$, $f(W)$ is necessarily of general type. Moreover, since $W \not\subset \Sp_{alg}\left(X \slash Y \right)$, there exists at least one $y \in f(W)$ such that the fibre $W_y$ is of general type. By Theorem \ref{generization general type}, the morphism $f_{|W} \colon W \to f(W)$ is of general type. But $f(W)$ is of general type, hence  $W$ is of general type thanks to Proposition \ref{composition general type}. \\
Assume now that $g$ is finite. Then the image of a non-constant holomorphic map $\bC \to X$ is contained in a fibre of $f$ if and only if it is contained in a fibre of the composition $h$. This shows the equality $\Sp_h \left(X \slash Z \right) = \Sp_h \left(X \slash Y \right)$. The proofs for $\ast = ab$ and $\ast = alg$ are similar.
\end{proof}

\begin{prop}\label{special sets-union}
Fix $\ast \in \{h,ab, alg \}$. Let $f \colon X \to Y$ be a proper morphism of complex algebraic varieties. Let $X = \cup_i X_i$ be the decomposition of $X$ in irreducible components. Then 
\begin{enumerate}
    \item $\Sp_\ast(X \slash Y) = \cup_i \Sp_\ast(X_i \slash Y)$. 
    \item $\Sp_\ast(X \slash Y)$ is not Zariski dense in $X$ if and only if $\Sp_\ast(X_i \slash Y)$ is not Zariski dense in $X_i$ for some $i$.
    \end{enumerate}
\end{prop}

\begin{proof}
Left to the reader.
\end{proof}

\begin{prop}\label{special sets-normalization}
Fix $\ast \in \{h,ab, alg \}$. Let $X \to Y$ be a proper morphism of complex algebraic varieties. Let $\nu: \tilde{X} \to X$ be the normalization of $X$. Then, $\Sp_\ast(X \slash Y)$ is not Zariski dense in $X$ if and only if $\Sp_\ast(\tilde{X}\slash Y)$ is not Zariski dense in $\tilde{X}$.
\end{prop}

\begin{proof}
Thanks to Proposition \ref{special sets-union}, one can assume that $X$ is irreducible. Let $Z \subset X$ be the non-normal locus of $X$. Then $\nu \left(\Sp_\ast(\tilde{X} \slash Y) \right) \subset \Sp_\ast(X \slash Y)$ since $\nu$ is finite, and $\Sp_\ast(X \slash Y) \subset \nu\left(\Sp_\ast(\tilde{X} \slash Y)\right) \cup Z$ since $\nu$ is an isomorphism onto its image outside $\nu^{-1}(Z)$. The result follows.
\end{proof}

\begin{thm}\label{key result-closed immersion}
Let $S$ be an integral complex algebraic variety. Let $A \to S$ be an abelian scheme. Let $X \to S$ be an $A$-torsor over $S$. Let $Y \to X$ be a closed immersion. Assume that the generic fibre of the composite morphism $Y \to S$ is geometrically integral. If  $ Y \to S$ is of general type, then $\Sp_h(Y / S)$ is not Zariski dense in $Y$.
\end{thm}

\begin{proof}
Thanks to Corollary \ref{base change-special sets}, it is sufficient to prove the result after base-change by a dominant étale morphism $S^\prime \to S$. Therefore, one can assume from the beginning that $X \to S$ admits a section, that one can use to identify $X$ with $A$. For every $s \in S$, $Y_s$ is a closed subvariety of the abelian variety $A_s$. Then, by results of Kawamata \cite{Kawamata-Bloch}, $\Sp_h(Y_s)$ is Zariski-closed in $A_s$ and it is equal to the union of those translates of abelian subvarieties of $A_s$ that are contained in $Y_s$. Therefore, by \cite[Théorème 1.2]{McQuillan} (see also \cite{Abramovich}), the subset $\Sp_h(Y / S)$ is Zariski-closed in $Y$. If $f\colon Y \to S$ is of general type, then for a very general $s \in S$ the fibre $Y_s$ is of general type thanks to Lemma \ref{Kodaira dimension-general versus generic} (note that $Y\to S$ is projective). It follows from \cite[Theorem 4]{Kawamata-Bloch} that $\Sp_h(Y_s) \neq Y_s$ for a very general $s \in S$. Therefore, $\Sp_h(Y / S) \neq Y$ and the claim follows since $\Sp_h(Y / S)$ is Zariski-closed in $Y$.
\end{proof}

\section{Structural results on subschemes of abelian schemes}

\subsection{Descent of abelian subschemes}
\begin{thm}\label{descending abelian subschemes}
Let $T \to S$ be a dominant morphism between two irreducible complex algebraic varieties. Let $A \to S$ be an abelian scheme. Let $B \to T$ be an abelian subscheme of $A_T \to T$. Then, there exists an irreducible complex algebraic variety $U$, a dominant quasi-finite morphism $U \to S$, and an abelian subscheme $C \to U$ of $A_U \to U$ such that, letting $V = U \times_S T$, the abelian $V$-schemes $C_V$ and $B_V$ are equal when naturally seen as abelian subschemes of the abelian $V$-scheme $A_V$. 
\end{thm}
\begin{proof}
This is a direct consequence of \cite[Corollary 20.4]{Milne}.
\end{proof}

\subsection{Ueno fibration for closed subschemes of torsors under an abelian scheme}
Let $S$ be a complex algebraic variety. Let $A \to S$ be an abelian scheme. Let $X \to S$ be an $A$-torsor over $S$. Let $Y \subset X$ be a closed subscheme. The group functor that associates to every $T \to S$ the set $ \{ a \in A(T)\, | \,a+ Y(T) \subset Y(T) \}$ is represented by a closed group subscheme $N_A(Y)$ of $A$, called the normalizer of $Y$ in $A$, see \cite[II.1.3.6]{Demazure-Gabriel} or \cite[VIB.6.2.4]{SGA3-5et6}.

\begin{thm}\label{relative Ueno fibration}
Let $S$ be a complex algebraic variety. Let $A \to S$ be an abelian scheme. Let $X \to S$ be an $A$-torsor over $S$. Let $Y \subset X$ be a closed subscheme. Assume that $S$ is integral and that the generic fibre of $Y \to S$ is geometrically integral. Let $B \subset A$ be the identity component of the normalizer of $Y$ in $A$. Then, up to shrinking $S$, the group subscheme $B \subset A$ is an abelian subscheme such that, letting $Z$ denote the image of $Y$ in the quotient $X \slash B$, 
\begin{enumerate}
\item the commutative diagram
\[
\begin{tikzcd}
Y \arrow[d] \arrow[r, hookrightarrow] & X \arrow[d] \\
Z \arrow[r, hookrightarrow]& X \slash B
\end{tikzcd}
\]
is  cartesian.
\item the proper morphism $Z \to S$ is of general type.
\end{enumerate}
\end{thm}

Since the morphism $X \to X \slash B$ is a torsor under the abelian scheme $B_{X \slash B}$, the morphism $Y \to Z$ is a torsor under the abelian scheme $B_Z$.\\

When $S$ is a point, Theorem \ref{relative Ueno fibration} is a well-known result of Ueno \cite[Theorem 10.3]{Ueno-book}. The proof in the general case consists in showing that Ueno's construction can be made in family.

\begin{proof}
Since the generic fibre of $Y \to S$ is geometrically integral, up to shrinking $S$, one can assume that the fibre $Y_s$ is integral for every $s \in S$. Up to shrinking further $S$, one can assume that $N_A(Y)$ is a flat group subscheme of $A$, so that its identity component $B$ is an abelian subscheme of $A$. Let $Z$ denote the image of $Y$ in the quotient $X \slash B$, so that the first item holds. Moreover, since the construction we have done recover at any geometric generic point Ueno's construction, it follows that the morphism $Z \to S$ is of general type. 
\end{proof}

\begin{cor}\label{corollary Ueno fibration}
Let $S$ be a complex algebraic variety. Let $A \to S$ be an abelian scheme. Let $X \to S$ be an $A$-torsor over $S$. Let $W \to X$ be a finite morphism. Assume that $S$ is integral and that the generic fibre of $W \to S$ is geometrically integral. Then, up to shrinking $S$, there exists $B \to S$ an abelian subscheme of $A$ such that if we let $Y \subset X$ be the image of $W \to X$ and $Z \subset X/B$ be the image of $W \to X \slash B$, then 
\begin{enumerate}
\item $Z \to S$ is a proper morphism of general type,
\item the generic fibre of $Z \to S$ is geometrically integral,
\item the morphism $W \to Y$ is finite surjective,
\item and $Y \to Z$ is a torsor under the abelian scheme $B_Z$.
\end{enumerate}
\end{cor}

\[
\begin{tikzcd}
W \arrow[rd] \arrow[r] & Y \arrow[d] \arrow[r, hookrightarrow] & X \arrow[d] \\
  & Z \arrow[r, hookrightarrow]& X \slash B
\end{tikzcd}
\]

\subsection{A finiteness result}
 
\begin{thm}\label{A finiteness result}
Let $S$ be a complex algebraic variety. Let $A \to S$ be an abelian scheme. Let $X \to S$ be an $A$-torsor over $S$. For every closed subvariety $Y \subset X$, there exists finitely many triples $(\psi_i \colon T_i \to S, B_i)$, $i \in I$, where 
\begin{itemize}
\item $T_i$ is a complex algebraic variety,
\item $\psi_i$ is a morphism, and
\item $B_i \subset A_{T_i}$ is an abelian subscheme of $A_{T_i}$, 
\end{itemize} 
such that the following holds: for every $s \in S$, every abelian subvariety $C$ of $A_s$ and every $C$-orbit $P \subset X_s$, if $P \subset Y_s$, then $P$ is contained in an orbit of $(B_i)_t$ for some $i \in I$ and some $t \in \psi_i^{-1}(s)$. 
\end{thm}

\begin{proof}
The proof goes by induction on the dimension on $Y$, the case where $\dim Y = 0$ being trivial. \\
By decomposing $Y$ as the union of its irreducible components, one sees that it is sufficient to prove the case where $Y$ is irreducible. Since the composition $Y \to X \to S$ is proper, one can replace $S$ by the image of $Y$ and therefore assume that the composition $Y \to X \to S$ is surjective. In particular $S$ is irreducible.\\
One can also freely replace $S$ by a non-empty open subset, since the preimage of the complementary will have smaller dimension. Thanks to Lemma \ref{reduce to geometrically integral} below, there exists a non-empty open subset $U \subset S$ and $V \to U$ a surjective finite étale morphism such that, letting $f^\prime \colon Y_V \to V$ be the base-change, for every irreducible component $E$ of $Y_V$ the generic fibre of $f^\prime_{|E} \colon E \to V$ is geometrically integral. Observe that $\dim E \leq \dim Y$. Clearly, if we know that the conclusion of Theorem \ref{A finiteness result} holds for every $f^\prime_{|E} \colon E \to V$, then it will hold for $Y_U \to U$. Therefore we can assume that the generic fibre of $Y \to S$ is geometrically integral. \\
Thanks to Theorem \ref{relative Ueno fibration}, up to shrinking $S$ again, there exists an abelian subscheme $B \subset A$ such that, if we let $Z$ be the image of $Y$ in $X \slash B$, then 
\begin{enumerate}
\item $Z \to S$ is a proper morphism of general type,
\item the generic fibre of $Z \to S$ is geometrically integral.
\end{enumerate}
If $C$ is an abelian variety contained in a fibre of $Y \to S$, then either it is contained in a fibre of $Y \to Z$, or it is contained in the preimage of $\Sp_{ab}(Z/S) $. In the first case, $C$ is a fortiori contained in a fibre of the $B_{X \slash B}$-torsor $X \to X \slash B$ and we are done. In the second case, since $Z \to S$ is of general type, it follows from Theorem \ref{key result-closed immersion} that $\Sp_{ab}(Z/S)$ is not Zariski-dense in $Z$, hence we are done by induction.
\end{proof}

\begin{lem}\label{reduce to geometrically integral}
Let $f\colon X \to Y$ be a morphism of complex algebraic varieties. Assume that $Y$ is integral. There exists a dense open subset $U \subset Y$ and $V \to U$ a surjective finite étale morphism such that, letting $f^\prime \colon X_V \to V$ be the base-change, for every irreducible component $E$ of $X_V$ the generic fibre of $f^\prime_{|E} \colon E \to V$ is geometrically integral. 
\end{lem}
\begin{proof}
Thanks to \cite[\href{https://stacks.math.columbia.edu/tag/0551}{Tag 0551}]{stacks-project}, there exists a non-empty open subset $U \subset Y$ and $V \to U$ a surjective finite étale morphism such that, letting $f^\prime \colon X_V \to V$ be the base-change, all irreducible components of the generic fibre $F$ of $f^\prime$ are geometrically irreducible. Since $X$ is reduced, $X_V$ and $F$ are reduced. Let $F = Z_{1,\eta} \cup \cdots \cup Z_{n,\eta}$ be the decomposition in irreducible components of the generic fibre of $f^\prime$. Let $Z_i$ be the closure of $Z_{i,\eta}$ in $X_V$.
By applying \cite[\href{https://stacks.math.columbia.edu/tag/054W}{Tag 054W}]{stacks-project} to $X \backslash (Z_1 \cup \cdots \cup Z_n)$, there exists a non-empty open subset $V^\prime \subset V$ such that $X_{V^\prime} = (Z_1)_{V^\prime} \cup \cdots \cup (Z_n)_{V^\prime}$. Therefore, up to replacing $V$ by $V^\prime$, and $U$ by the image of $V^\prime$ under the finite étale morphism $V \to U$, we get that for every irreducible component $E$ of $X_V$ the generic fibre of $f^\prime_{|E} \colon E \to V$ is geometrically irreducible and reduced, hence geometrically integral since we are in characteristic zero, cf. \cite[\href{https://stacks.math.columbia.edu/tag/020I}{Tag 020I}]{stacks-project}.
\end{proof}

\begin{cor}\label{entire curves contained in a proper subvariety}
Let $S$ be a complex algebraic variety. Let $A \to S$ be an abelian scheme. Let $X \to S$ be an $A$-torsor over $S$. Let $E$ be a relative effective Cartier divisor in $X \slash S$.  Then there exists finitely many pairs $(\psi_i \colon T_i \to S, B_i)$, $i \in I$, where 
\begin{itemize}
\item $T_i $ is an irreducible complex algebraic variety,
\item $\psi_i$ is a morphism,
\item $B_i \subsetneq A_{T_i}$ is an abelian subscheme of $A_{T_i}$, 
\end{itemize} 
such that for every $s \in S$ and every non-constant holomorphic map $f \colon \bC \to X_s$ whose image is contained \textit{up to translation} in $E_s$, the image of $f$ is contained in an orbit of $(B_i)_t$ for some $i \in I$ and some $t\in \psi_i^{-1}(s)$.
\end{cor}

\begin{proof}
Indeed, the Zariski closure of $f(\bC)$ is an orbit of an abelian variety $C \subset A_s$ by the Bloch-Ochiai Theorem \cite[Theorem 3]{Kawamata-Bloch}, which is contained up to translation in a fibre of $E \to S$. Hence one can apply Theorem \ref{A finiteness result}.
\end{proof}


\section{A truncated Second Main Theorem for torsors over an abelian scheme}

The goal of this section is to explain a proof of Theorem \ref{truncated SMT for abelian schemes}.

\subsection{Notations}\label{Notations Nevanlinna}

For every positive real number $r$, we set $\Delta(r) := \{ z \in \bC \, | \, |z| < r \}$. If $\alpha$ is a current of type $(1,1)$ on $\bC$, one defines its \emph{Nevanlinna characteristic function} by
\[T(r, \alpha)=\int_1^r\left(\int_{\Delta(t)}\alpha\right)\frac{dt}{t}.\]

If $X$ is a smooth proper complex algebraic variety, $L$ a line bundle on $X$ and $f\colon \bC \to X$ a non-constant holomorphic map, then one defines the \emph{characteristic function of $f$ with respect to $L$} to be the function
\[T(r,f,L):=T(r, f^*C_1\left(L,h\right)).\]
Here $C_1(L,h)$ is the first Chern form of $L$ with respect to a $\mathscr{C}^{\infty}$ metric $h$ on $L$ (the notation is somewhat abusive since the function $T \left(r,f,L \right)$ depends on the choice of the metric $h$, but only up to a bounded function, see \cite[Lemma 3.1]{Yamanoi-lectures}). If $D$ is a Cartier divisor on $X$, we let $T (r, f, D) := T (r, f, \cO_X (D)) $.\\

If $X$ is a complex algebraic variety, $Z \subset X$ a (possibly non-reduced) closed analytic subspace and $f\colon \bC \to X$ a non-constant holomorphic map such that $f(\bC)\not\subset Z$, then one defines the \emph{counting function of $f$ with respect to $Z$} to be the function
\[ N(r,f,Z):=\int_1^r\left(\sum_{z \in \Delta(t)}  \mathrm{ord}_z \left(f^{-1}\left(Z\right) \right) \right)\frac{dt}{t}.\]
Here $f^{-1}\left(Z\right) $ denotes the (possibly non-reduced) inverse image of $Z$. \\

For every integer $k \geq 1$, we define the \emph{k-th truncated counting function with respect to $Z$} by
\[ N^{(k)}(r,f,Z):=\int_1^r\left(\sum_{z \in \Delta(t)}  \min \left( k , \mathrm{ord}_z \left(f^{-1}\left(Z\right) \right) \right) \right)\frac{dt}{t}.\]

\subsection{Auxiliary results}
For the reader convenience, we recall some results that we will use in the sequel.
\begin{thm}[Nevanlinna inequality, see {\cite[Corollary 3.3]{Yamanoi-lectures}}]\label{Nevanlinna's inequality}
Let $X$ be a smooth projective complex algebraic variety and let $D$ be an effective Cartier divisor on $X$. Let $f \colon \bC \to X$ be a holomorphic curve such that $f (\bC) \not \subset  D$. Then
\[ N (r, f, D) \leq T (r, f, D) + O(1). \]
In particular, 
\[T (r, f, D) + O(1) \geq 0.\]
\end{thm} 

\begin{prop}
If $X$ is a complex algebraic variety, $Z \subset Y \subset X$ two closed subschemes and $f\colon \bC \to X$ a non-constant holomorphic map such that $f(\bC)\not\subset Y$, then
\[ N (r, f, Z)  \leq    N (r, f, Y)  . \] 
\end{prop}

\begin{proof}
This follows immediately from the inclusions
$f^{-1}(Z) \subset f^{-1}(Y) \subset \bC$.
\end{proof}

\begin{prop}
Let $X$ be a complex algebraic variety and $Z \subset X$ be a closed subscheme. Let $f \colon \bC \to X$ be a non-constant holomorphic map such that $f(\bC) \not \subset Z$. If $f(\Delta_k) \subset Z$ for some positive integer $k$ (where $\Delta_k$ is seen as a closed analytic subspace of $\bC$), then $\mathrm{ord}_0 \left(f^{-1} \left(Z \right)\right) \geq k + 1$.
\end{prop}

\subsection{A Second Main Theorem for torsors over an abelian scheme}

We first prove a Second Main Theorem for torsors over an abelian scheme, building on a strategy explained in \cite{Yamanoi-lectures}. 

\begin{thm}\label{SMT for abelian schemes}
Let $S$ be a complex algebraic variety. Let $A \to S$ be an abelian scheme. Let $X \to S$ be an $A$-torsor over $S$. Let $L \to X$ be a line bundle which is relatively ample over $S$. Let $D \subset X$ be a relative effective Cartier divisor. Assume that $D_s$ is integral for every $s \in S$. \\
Then, there exists an integer $\rho \geq 1$ such that, for every $s\in S$ and every orbit $f \colon \bC \to X_s$ of a one-parameter subgroup of $A_s$, one has either that
\begin{enumerate}
    \item the image of a translate of $f$ is contained in $D_s$, or
    \item $T (r, f, D_s) \leq  N^{(\rho)}(r, f, D_s) +o\left(T \left(r,f,L_s \right)\right) ||$.
\end{enumerate}
\end{thm}

The proof of Theorem \ref{SMT for abelian schemes} is given in the remainder of this section.

\begin{prop}\label{Proposition-uniformity-jets}
Let $S$ be a complex algebraic variety. Let $A \to S$ be an abelian scheme. Let $X \to S$ be an $A$-torsor over $S$. Let $W \subset X$ be a closed subscheme. Then, there exists an integer $\rho \geq 1$ such that for every $s\in S$ and every orbit $f \colon \bC \to X_s$ of a one-parameter subgroup of $A_s$, if $J_\rho (f)(0) \in J_\rho (W/S)$, then $f(\bC) \subset W$.
\end{prop}
\begin{proof}
For every integer $k \geq 1$, we have a commutative diagram
\[
\begin{tikzcd}
J_k (W /S) \arrow[r, hookrightarrow] \arrow[d] & J_k (X /S) = X \times_S \Lie^k (A /S) \arrow[d]  \\
J_1 (W /S)  \arrow[r, hookrightarrow] & J_1 (X /S) = X \times_S \Lie (A /S).
\end{tikzcd}
\]

As in section \ref{Canonical sections of the forget maps}, let $\sigma_k \colon \Lie (A /S) \to \Lie^k (A/S)$ denote the canonical section of the natural projection $\Lie^k (A/S) \to \Lie (A/S)$, so that one has an induced section $\mathrm{Id} \times \sigma_k \colon J_1 (X/S) \to J_k (X/S) $ of the natural projection $J_k (X/S) \to J_1 (X/S) $. Every point of $J_1 (X/S) $ corresponds canonically to an orbit $f \colon \bC \to X_s$ of a one-parameter subgroup of $A_s$, and $\mathrm{Id} \times \sigma_k$ sends $f$ to the $k$-jet $J_k(f)(0)$. \\
Write $W_k := \left( \mathrm{Id} \times \sigma_k \right)^{-1} (J_k (W/S))$, so that the $W_k$'s form a descending chain of closed subschemes $W_{k+1} \subset W_k \subset \ldots \subset J_1 (X/S)$. By Noetherianity, there exists an integer $\rho \geq 1$ such that $W_k = W_\rho$ for any $k\geq \rho$.\\
By construction, for every $s\in S$, if an orbit $f \colon \bC \to X_s$ of a one-parameter subgroup of $A_s$ satisfies $J_\rho (f) (0) \in J_\rho (W/S)$, then $J_k (f) (0) \in J_k (W/S)$ for every integer $k \geq 1$, so that $f(\bC) \subset W$.
\end{proof}

\begin{cor}\label{Corollary-uniformity-jets}
Let $S$ be a complex algebraic variety. Let $A \to S$ be an abelian scheme. Let $X \to S$ be an $A$-torsor over $S$. Let $W \subset X$ be a closed subscheme. Then, there exists an integer $\rho \geq 1$ such that, for every $s\in S$ and every orbit $f \colon \bC \to X_s$ of a one-parameter subgroup of $A_s$, if $p_\rho \left( J_\rho (f) (0) \right) \in p_\rho \left( J_\rho (W/S) \right)$, then a translate $f_a$ of $f$ satisfies $f_a(\bC) \subset W_s$.
\end{cor}
Here and below, for every integer $k \geq 1$, $p_k \colon  J_k (X/S) \to \Lie^k (A/S)$ denote the projection on the second factor in the canonical decomposition $J_k (X /S) = X \times_S \Lie^k (A /S)$.

\begin{proof}
This follows immediately from the Proposition \ref{Proposition-uniformity-jets}, by observing that for every integer $k \geq 1$, every $s\in S$ and every orbit $f \colon \bC \to X_s$ of a one-parameter subgroup of $A_s$, the condition $p_k \left(J_k(f)(0) \right) \in p_k \left( J_k (W/S) \right)$ is equivalent to the existence of a translate $f_a$ of $f$ satisfying $J_k(f_a)(0) \in J_k (W/S) $.
\end{proof}

\begin{proof}[Proof of Theorem \ref{SMT for abelian schemes}]
We keep the notations from the statement. Thanks to Corollary \ref{Corollary-uniformity-jets}, there exists an integer $\rho \geq 1$ such that for every $s\in S$ and every orbit $f \colon \bC \to X_s$ of a one-parameter subgroup of $A_s$, if $p_\rho \left( J_\rho (f) (0) \right) \in p_\rho \left( J_\rho (D/S) \right)$, then the image of a translate of $f$ is contained in $D_s$. On the other hand, let $s\in S$ and $f \colon \bC \to X_s$ be an orbit of a one-parameter subgroup of $A_s$ such that $p_\rho \left( J_\rho (f) (0) \right) \notin p_\rho \left( J_\rho (D/S) \right)$. Since $p_\rho \left( J_\rho (D_s) \right)$ is a closed subscheme of the finite-dimensional complex vector space $\Lie^\rho (A_s)$, there exists an algebraic function in $\HH^0( \Lie^\rho (A_s), \cO)$ that vanishes on $p_\rho \left( J_\rho (D_s) \right)$ but not on $p_\rho \left( J_\rho (f)(0) \right)$. Its pull-back along the projection $(p_\rho)_s \colon J_\rho (X_s) \to \Lie^\rho (A_s)$ yields $\omega \in \HH^0(J_\rho (X_s), \cO)$ that vanishes on $J_\rho (D_s)$ but not on $J_\rho (f)(0)$.
Theorem \ref{SMT for abelian schemes} is therefore a consequence of Theorem \ref{half SMT} below.
\end{proof}

\begin{thm}[cf. {\cite[Theorem 3.13]{Yamanoi-lectures}}]\label{half SMT}
Let $\tX$ be a smooth projective complex variety. Let $\tL$ be an ample line bundle on $\tX$. Let $\tD \subset \tX$ be an integral divisor. Let $\rho \geq 1$ be an integer and $\omega \in \HH^0(J_\rho (\tX), \cO)$ be an algebraic function that vanishes on $J_\rho (\tD)$. Then, if $f \colon \bC \to \tX$ is a non-constant holomorphic map with $f(\bC) \not \subset \tD$, then it satisfies either
\begin{enumerate}
    \item $\omega(J_\rho (f)) = 0$, or
    \item  $T (r, f, \tD) \leq  N^{(\rho)}(r, f, \tD) +o\left(T(r,f,\tL)\right) ||$.
\end{enumerate}
\end{thm}

\subsection{Intersection with subvarieties of codimension at least $2$}

A key step in the proof of Theorem \ref{truncated SMT for abelian schemes} consists in the following estimate of the Nevanlinna counting function of an entire curve with respect to a subvariety of codimension at least $2$, whose proof is given in the remainder of this section.

\begin{thm}\label{codim 2}
Let $S$ be a complex algebraic variety. Let $A \to S$ be an abelian scheme. Let $X$ be an $A$-torsor over $S$. Let $L \to X$ be a line bundle which is relatively ample over $S$. Let $Y \subset X$ be a closed subscheme such that $\codim(Y_s, X_s) \geq 2$ for every $s \in S$.\\
For every $\epsilon >0$, there exists a complex algebraic variety $T$, a morphism $\psi \colon T \to S$ and a relative effective Cartier divisor $E$ in $X_T \slash T$ such that, for every $s\in S$ and every orbit $f \colon \bC \to X_s$ of a one-parameter subgroup of $A_s$ such that $f(\bC) \not \subset Y$, one has either 
\begin{enumerate}
    \item $  N^{(1)}(r, f, Y_s) \leq \epsilon \cdot T(r,f,L_s) || $, or
    \item $f( \bC)$ is contained in $E_t$ for some $t \in \psi^{-1}(s)$. (Here, we identify $X_t$ with $X_s$ using $\psi$.)
\end{enumerate}
\end{thm}

Let $S$ be a complex algebraic variety. Let $A \to S$ be an abelian scheme. Let $X \to S$ be an $A$-torsor over $S$. Let $L \to X$ be a line bundle which is relatively ample over $S$. Let $M \to J_1(X / S)$ denote the pull-back of $L \to X$ along the morphism $J_1(X / S) \to X$. Since there is a cartesian square

\[
\begin{tikzcd}
J_1(X / S) \arrow[d] \arrow[r] & \Lie(A/S) \arrow[d] \\
X \arrow[r]& S,
\end{tikzcd}
\]
the line bundle $M \to J_1(X / S)$ is relatively ample over $\Lie(A/ S)$. For every $s \in S$ and every holomorphic map $f \colon \bC \to X_s$, there is a commutative diagram
\[
\begin{tikzcd}
 & J_1(X_s) \arrow[d]  \\
\bC \arrow[r, "f"] \arrow[ur, "J_1(f)"]& X_s,
\end{tikzcd}
\]

from what it follows that 
\begin{equation}\label{relating T for f and J_1(f)}
T(r,f,L_s) = T(r,J_1(f),M_s) .
\end{equation}

Let $Y \subset X$ be a closed subscheme such that $\codim(Y_s, X_s) \geq 2$ for every $s \in S$. Write $Z := Y \times_X J_1(X/S) = Y \times_S \Lie(A/S)$, so that $Z \subset J_1(X/S)$ is a closed subscheme such that $\codim(Z_t, J_1(X/S)_t) \geq 2$ for every $t \in \Lie(A/S)$.
For every $s \in S$ and every holomorphic map $f \colon \bC \to X_s$ such that $f(\bC) \not \subset Y$, one has
\begin{equation}\label{relating N for f and J_1(f)}
 N^{(1)}(r, f, Y_s)  = N^{(1)}(r, J_1(f), Z_s). 
\end{equation}

With the preceding notations, and in view of (\ref{relating T for f and J_1(f)}) and (\ref{relating N for f and J_1(f)}), Theorem \ref{codim 2} reduces to the following result.
\begin{thm}\label{codim 2, bis}
Let $S$ be a complex algebraic variety. Let $A \to S$ be an abelian scheme. Let $X \to S$ be an $A$-torsor over $S$. Let $M \to J_1(X/S)$ be a line bundle which is relatively ample over $\Lie(A/S)$. Let $Z \subset J_1(X/S)$ be a closed subscheme such that $\codim(Z_t, J_1(X/S)_t) \geq 2$ for every $t \in \Lie(A/S)$.\\
For every $\epsilon >0$, there exists a complex algebraic variety $T$, a morphism $\psi \colon T \to \Lie(A/S)$ and a relative effective Cartier divisor $E$ in $X_T \slash T$ such that, for every $s \in S$ and every orbit $f \colon \bC \to X_s$ of a one-parameter subgroup of $A_s$ such that $J_1(f)(\bC) \not \subset Z$, one has either 
\begin{enumerate}
    \item  $  N^{(1)}(r, J_1(f), Z_s) \leq \epsilon \cdot T(r,J_1(f),M_s) || $, or
    \item $J_1(f)( \bC)$ is contained in $E_t$ for some $t \in \phi^{-1}(s)$, where $\phi$ is the composition of $\psi$ with $\Lie(A/S) \to S$. (Here again, we identify $X_t$ with $X_s$ using $\phi$.)
\end{enumerate}
\end{thm}

Let us now prove Theorem \ref{codim 2, bis}. For every integer $l \geq 1$, by definition of the scheme $J_l(A / S)$ of $l$-jet differentials of $A$ over $S$, there is a tautological $S$-morphism $ J_l (A / S) \times_\bC \Delta_l \to A$. By precomposing with the canonical $S$-morphism $\Lie (A / S) \to J_l(A / S)$ from section \ref{Canonical sections of the forget maps}, one obtains the $S$-morphism $ \Lie (A / S) \times_\bC \Delta_l \to A$. Together with the $S$-morphism $ \Lie (A / S) \times_\bC \Delta_l \to \Lie (A / S)$ projection on the first factor, this yields a commutative diagram:
\[
\begin{tikzcd}
\Lie (A / S) \times_\bC \Delta_l \arrow[rd] \arrow[r] & J_1 (A / S) \arrow[d] \arrow[r] & A \arrow[d] \\
& \Lie (A / S) \arrow[r] & S
\end{tikzcd}
\]

Since the composite morphism $\Lie (A / S) \times_\bC \Delta_l \to J_1 (A / S) \to \Lie (A / S)$ is proper, the morphism $\Lie (A / S) \times_\bC \Delta_l \to J_1 (A / S)$ is proper too.\\

Since $J_1(X/S)$ is a $J_1(A/S)$-torsor over $\Lie(A/S)$, the composition of the action 
\[J_1(A/S) \times_{\Lie(A/S)} J_1(X/S) \to J_1(X/S) \]
with the $\Lie (A / S)$-morphism 
\[ \left( \Lie (A / S) \times_\bC \Delta_l \right)\times_{\Lie(A/S)} Z \to J_1(A/S) \times_{\Lie(A/S)} J_1(X/S) \]
induces a proper $\Lie (A / S)$-morphism
\[ \left( \Lie (A / S) \times_\bC \Delta_l \right) \times_{\Lie(A/S)} Z \to J_1(X/S). \]
Let $Z_l \subset J_1 (X/S)$ denote its image, and let 
\[ \tilde{Z_l} := \left( \Lie (A / S) \times_\bC \Delta_l \right) \times_{\Lie(A/S)} Z = Z \times_\bC \Delta_l, \]
so that one has the following commutative  diagram: 
\[
\begin{tikzcd}
 \tilde{Z_l} \arrow[rrd] \arrow[r, twoheadrightarrow ] &  Z_l \arrow[r, hookrightarrow] & J_1(X/S) \arrow[d] \\
& & \Lie(A/S).
\end{tikzcd}
\]

\begin{prop}\label{Auxiliary divisor}
Let $\epsilon > 0$. Let $W \subset \Lie(A/S)$ be a closed integral subscheme.
Then, there exists two integers $k,l \geq 1$ with $\frac{k}{l} < \epsilon$, $U \subset W $ a Zariski-dense open subset and a relative effective Cartier divisor $E \subset J_1(X/S)_{U} $ over $U$ defined by a section in $\HH^0\left( J_1(X/S)_{U}, M^{\otimes k} \right) $ such that $(Z_l)_U \subset E$.
\end{prop}

\begin{proof}
We denote by $\tZ, \tZ_l, \tilde{\tZ}_l, \tX$ the restriction of $Z, Z_l, \tilde{Z}_l, J_1(X/S)$ respectively to the generic point of $W$. Note that $\tilde{\tZ}_l = \tZ \times_\bC \Delta_l$.  \\

Let us first prove that for every positive integers $k,l$:
\[ h^0(\tZ_l, M_{|\tZ_l} ^{\otimes k}) \leq (l+ 1) \cdot h^0(\tZ, M_{|\tZ} ^{\otimes k}).  \]
By construction, there is a schematically surjective morphism $ \tilde{\tZ}_l  \to \tZ_l$, hence an injection 
\[   \HH^0(\tZ_l, M_{|\tZ_l} ^{\otimes k}) \hookrightarrow \HH^0(\tilde{\tZ}_l, M_{|\tilde{\tZ}_l} ^{\otimes k}).  \]
Therefore, it is sufficient to prove that for every positive integers $k,l$:
\[ h^0(\tilde{\tZ}_l, M_{|\tilde{\tZ}_l} ^{\otimes k}) \leq (l+ 1) \cdot h^0(\tZ, M_{|\tZ} ^{\otimes k}).  \]
This is a consequence of the following proposition.

\begin{prop}
Let $V$ be a scheme proper over a field $K$. For every positive integer $l$, let $V_l := V \times_K \Spec \left( K[t] / (t^{l + 1}) \right)$. For any line bundle $M$ on $V_l$, 
\[ h^0(V_l, M) \leq (l+ 1) \cdot h^0(V_0, M_{|V_0}).  \]
\end{prop}
\begin{proof}
For every non-negative integers $k <l$, denote by $\iota_k ^l \colon V_k \hookrightarrow V_l$ the canonical closed immersion. In particular, for every positive integer $l$, we have an exact sequence of $\cO_{V_l}$-modules:
\[  0 \to \cI_{l-1}^l \to \cO_{V_l} \to (\iota_{l-1}^l)_\ast \cO_{V_{l-1}}\to 0 . \] 

By tensoring with the locally-free $\cO_{V_l}$-module $M$, we get an exact sequence of $\cO_{V_l}$-modules:
\begin{equation}\label{exact sequence thickening}
0 \to \cI_{l-1}^l \otimes_{\cO_{V_l}} M \to M \to \left(\left(\iota_{l-1}^l\right)_\ast \cO_{V_{l-1}} \right) \otimes_{\cO_{V_l}} M \to 0.
\end{equation}

On the one hand, since $\cI_{l-1}^l = (t^l)$ is a principal ideal sheaf of $\cO_{V_l}$ which is annihilated by the ideal sheaf $\cI_0^l = (t)$, there is an isomorphism $\cI_{l-1}^l \simeq (\iota_0^l)_\ast \cO_V$ as $\cO_{V_l}$-modules. Therefore, $\cI_{l-1}^l \otimes_{\cO_{V_l}} M \simeq (\iota_0^l)_\ast \left( M_{|V} \right)$ as $\cO_{V_l}$-modules.
On the other hand, $\left(\left(\iota_{l-1}^l\right)_\ast \cO_{V_{l-1}} \right)\otimes_{\cO_{V_l}} M = (\iota_{l-1}^l)_\ast \left( M_{|V_{l-1}} \right)  $. Therefore, it follows from the exact sequence (\ref{exact sequence thickening}) that 
\[ h^0(V_l, M) \leq h^0(V, M_{|V}) + h^0(V_{l-1}, M_{|V_{l-1}}),  \]
and the result follows by induction on $l$.
 \end{proof}

Let $C$ be a positive number such that for every positive integer $k$:
\[ h^0(\tZ, M_{|\tZ} ^{\otimes k}) \leq C \cdot k^{\dim_{\bC(W)} \tZ },\]
so that for every positive integer $k$:
\[ h^0(\tZ_l, M_{|\tZ_l} ^{\otimes k}) \leq (l+ 1) \cdot h^0(\tZ, M_{|\tZ} ^{\otimes k}) \leq C \cdot (l+ 1) \cdot k^{\dim_{\bC(W)} \tZ }.\]
On the other hand, since $M$ is relatively ample, the line bundle $M_{|\tX}$ is ample on $\tX$. In particular, there exists a positive number $C^\prime$ such that for $k \gg 1$: 
\[ h^0(\tX, M_{|\tX}^{\otimes k}) \geq C^\prime \cdot k^{\dim_{\bC(W)} \tX}. \]
Therefore, there exists $k,l \geq 1$ with $\frac{k}{l} < \epsilon$, such that the natural $\bC(W)$-linear map
\[ \HH^0(\tX, M_{|\tX}^{\otimes k} )  \to \HH^0(\tZ_l, M^{\otimes k}_{|\tZ_l} ) \]
is not injective. Let $\texttt{E} \subset \tX$ be a Cartier divisor corresponding to a section in the kernel. Denote by $E$ the Zariski closure of $\texttt{E}$ in $J_1(X/S)$ (or equivalenty in $J_1(X/S)_W$). By generic flatness, there exists $U \subset W $ a Zariski-dense open subset such that $E \subset J_1(X/S)_{U} $ is a relative effective Cartier divisor over $U$. Moreover, up to shrinking $U$, the divisor $E$ corresponds to a section in $\HH^0 \left( J_1 (X/S)_{U}, M ^{\otimes k} \right) $ and $(Z_l)_U \subset E$. This finishes the proof of Proposition \ref{Auxiliary divisor}.
\end{proof}

\begin{prop}\label{induction step} For every $\epsilon > 0$, there exists
\begin{itemize}
\item a partition of $\Lie(A/S)$ in finitely many locally closed subvarieties $T_j$, $j \in J$, 
\item for every $j \in J$, a relative effective Cartier divisor $E_j \subset \left(J_1(X/S)\right)_{T_j}$ over $T_j$, 
\end{itemize}   
such that, for every $s \in S$ and every orbit $f \colon \bC \to X_s$ of a one-parameter subgroup of $A_s$ such that the image of the induced holomorphic map $J_1(f) \colon \bC \to J_1 (X_s)$ is not contained in one of the $E_j$, the following estimates holds:
\[  N^{(1)}(r, J_1(f), Z_s) \leq \epsilon \cdot T(r,J_1(f),M_s) ||  . \]
\end{prop}

\begin{proof} Fix $\epsilon > 0$. By induction from Proposition \ref{Auxiliary divisor}, there exists  a partition of $\Lie(A/S)$ in finitely many locally closed subvarieties $T_j$, $j \in J$, and for every $j \in J$ a relative effective Cartier divisor $E_j \subset \left(J_1 (X/S) \right)_{T_j}$ over $T_j$ such that there exists two integers $k_j,l_j \geq 1$ with $\frac{k_j}{l_j} < \epsilon$ such that the effective Cartier divisor $E_j \subset \left( J_1 (X/S) \right)_{T_j} $ corresponds to a section in $\HH^0 \left( \left( J_1(X/S) \right)_{T_j}, M ^{\otimes k_j} \right) $ and $\left(Z_{l_j} \right)_{T_j} \subset E_j$.\\

Let $f \colon \bC \to X_s$ be an orbit of a one-parameter subgroup of $A_s$. The image of the induced holomorphic map $J_1(f) \colon \bC \to J_1 (X_s)$ is contained in a fibre of $J_1(X / S) \to \Lie(A /S)$, hence it is contained in $\left( J_1 (X/S) \right)_{T_j}$ for some $j \in J$. Assume that $J_1(f)(\bC) \not \subset E_j$. Since $(Z_{l_j}) \subset E_j$, it follows that $N(r, J_1(f), \left( Z_{l_j}\right)_s) \leq N(r, J_1(f), \left(E_j\right)_s)$. Thanks to Nevanlinna's inequality (cf. Theorem \ref{Nevanlinna's inequality}), this implies 
\[ N(r, J_1(f), \left( Z_l \right)_s) \leq     T(r, J_1(f), M_s^{\otimes k_j}) + O(1) \leq k_j \cdot T(r, J_1(f), M_s) + O(1). \]

On the other hand, one has  
\[ (l_j + 1) \cdot N^{(1)}(r, J_1(f), Z_s) \leq  N(r, J_1(f), \left( Z_l \right)_s)\]
by definition of $Z_l$. Therefore, we get that 
 \[   (l_j + 1) \cdot N^{(1)}(r, J_1(f), Z_s)    \leq k_j \cdot T(r, J_1(f), M_s) + O(1), \]
  so that
  \[   N^{(1)}(r, J_1(f), Z_s) \leq  \frac{k_j}{l_j+ 1}  \cdot T(r, J_1(f), M_s)  + O(1). \]

Since $\frac{k_j}{l_j} < \epsilon$, the result follows by noting that $T(r, J_1(f), M_s)$ goes to infinity with $r$, since $J_1(f)$ is not constant and $M$ is ample in restriction to the fibre of $J_1(X/S) \to \Lie(A/S)$ that contains the image of $J_1(f)$.
\end{proof}

Letting $T$ be the disjoint union of the $T_i$'s, and $E$ be the disjoint union of the $E_i$'s, we see that Theorem \ref{codim 2, bis} is a consequence of Proposition \ref{induction step}.

\subsection{End of the proof of Theorem \ref{truncated SMT for abelian schemes}}

We will prove Theorem \ref{truncated SMT for abelian schemes} for every tuple $\left(S, A, X, D, L \right)$ by induction on the relative dimension of $X / S$. Here, $S$ is a complex algebraic variety, $A \to S$ is an abelian scheme, $X \to S$ is an $A$-torsor over $S$, $D \subset X$ is a reduced effective Cartier divisor and $L \to X$ is a line bundle which is relatively ample over $S$.\\
There is nothing to prove when the relative dimension is zero.  Let $d \geq 1$ be an integer and assume that Theorem \ref{truncated SMT for abelian schemes} is true for every tuple $\left(S, A, X, D, L \right)$ when in addition $X / S$ has relative dimension $<d$.  We decompose the proof in several steps.\\

We start with the following easy observation.
\begin{prop} \label{etale cover}
Let $T$ be a complex algebraic variety and $T \to S$ be an étale dominant morphism. Assume that Theorem \ref{truncated SMT for abelian schemes} holds for the tuple $\left(T, A_T, X_T, D_T, L_T \right)$. Then Theorem \ref{truncated SMT for abelian schemes} holds for the tuple $\left(S, A, X, D, L \right)$.
\end{prop}
\begin{proof}
Write the morphism $T \to S$ as the composition of a finite étale surjective morphism $T \to U$ and an open immersion $U \subset S$, and assume that Theorem \ref{truncated SMT for abelian schemes} is verified by the tuple $\left(T, A_T, X_T, D_T, L_T \right)$. For every $\epsilon >0$, let $\Xi_T = \Xi(A_T, X_T, D_T, L_T, \epsilon)  \subsetneq X_T$ be the closed subscheme whose existence is claimed in the statement. Then Theorem \ref{truncated SMT for abelian schemes} is verified by the tuple $\left(S, A, X, D, L \right)$, by letting $\Xi$ be the union of the closure of the image of $\Xi_T$ in $X$ and the preimage in $X$ of the complementary of $U$ in $S$.
\end{proof}

We will also use several times the following result.

\begin{prop} \label{key induction step}
Let $T$ be a complex algebraic variety equipped with a morphism $\phi \colon T \to S$. Let $E$ be a relative effective Cartier divisor in $X_T \slash T$. Then, for every $\epsilon > 0$, there exists $\Xi \subsetneq X$ a closed algebraic subvariety such that, for every $s\in S$ and every orbit $f \colon \bC \to X_s$ of a one-parameter subgroup of $A_s$ such that $f( \bC)\subsetneq \Xi$ and $f( \bC)$ is contained \textit{up to translation} in $E_t$ for some $t \in \phi^{-1}(s)$ (where $X_t$ is identified with $X_s$ using $\phi$), the estimate
\[ T(r, f, D_s) \leq N^{(1)}(r, f, D_s) +  \epsilon \cdot T(r, f, L_s) \, || \,  \]
holds. 
\end{prop}
\begin{proof}

Thanks to Corollary \ref{entire curves contained in a proper subvariety}, there exists finitely many pairs $(\psi_i \colon T_i \to T, B_i)$, $i \in I$, where $T_i $ is an irreducible complex algebraic variety, $\psi_i$ is a morphism, and $B_i \subsetneq A_{T_i}$ is an abelian subscheme of $A_{T_i}$, such that for every $t \in T$ and every non-constant holomorphic map $f \colon \bC \to X_t$ whose image is contained \textit{up to translation} in $E_t$, the image of $f$ is contained in an orbit of $(B_i)_u$ for some $i \in I$ and some $u \in \psi_i^{-1}(t)$.\\

For every $i$, let $\mu_i \colon T_i \to S$ be the composition of $\psi_i$ with $\phi$. Then, for every $s\in S$ and every orbit $f \colon \bC \to X_s$ of a one-parameter subgroup of $A_s$ such that $f( \bC)$ is contained \textit{up to translation} in $E_t$ for some $t \in \phi^{-1}(s)$, the image of $f$ is contained in an orbit of $(B_i)_u$ for some $i \in I$ and some $u \in \mu_i^{-1}(s)$.\\

Thanks to Theorem \ref{descending abelian subschemes} and an immediate induction, it follows that there exists finitely many pairs $(\nu_j \colon U_j \to S, C_j)$, $j \in J$, where $U_j$ is an irreducible complex algebraic variety, $\nu_j$ is a quasi-finite morphism, and $C_j \subsetneq A_{U_j}$ is an abelian subscheme of $A_{U_j}$, such that,  for every $s\in S$ and every orbit $f \colon \bC \to X_s$ of a one-parameter subgroup of $A_s$ such that $f( \bC)$ is contained \textit{up to translation} in $E_t$ for some $t \in \phi^{-1}(s)$, the image of $f$ is contained in an orbit of $(C_j)_u$ for some $j \in J$ and some $u \in \nu_j^{-1}(s)$.\\

Since for every $j \in J$ the relative dimension of $X_{U_j} \to X_{U_j} \slash C_j$ is smaller than $d$, the induction applies and Theorem \ref{truncated SMT for abelian schemes} is verified by the tuples $\left(X_{U_j} \slash C_j, C_j, X_{U_j}, D_{U_j}, L_{U_j} \right)$. Therefore, for every $\epsilon >0$, let $\Xi_{j} = \Xi\left(X_{U_j} \slash C_j, C_j, X_{U_j}, D_{U_j}, L_{U_j}, \epsilon \right)  \subsetneq X_{U_j}$ be the closed subscheme whose existence is claimed in the statement. Let $\Xi$ be the union of the closure of the images of the $\Xi_{j}$'s in $X$. Since the $\nu_i$'s are quasi-finite, $\Xi$ is distinct from $X$. By construction, $\Xi$ satisfies the stated property.
\end{proof}

We now explain how to complete the proof of Theorem \ref{truncated SMT for abelian schemes} by successive reductions. Thanks to Proposition \ref{etale cover}, one can assume that $S$ is integral and smooth, and that $D$ is flat over $S$. Moreover, one can assume that the $A$-torsor $X$ is trivial and identify $X$ with $A$.

\begin{claim} 
One can assume that $D$ is integral. 
\end{claim}
\begin{proof}
Let $D = \cup_{i= 1}^k \, D_i$ be the decomposition in irreducible components of $D$. Then 
\[ \sum_{i=1}^k N^{(1)}\left(r, f, (D_i)_s \right)  \leq N^{(1)}\left(r, f, D_s \right) + \sum_{i \neq j} N^{(1)}\left(r, f, \left(D_i \cap D_j \right)_s \right) \]
and
\[ T(r, f, D_s) \leq \sum_{i=1}^k T\left(r, f, \left(D_i\right)_s \right) + O(1). \]

By shrinking $S$, one can assume that $\codim( \left(D_i \cap D_j \right)_s,  X_s) \geq 2$ for every $i \neq j$ and every $s \in S$. By applying Theorem \ref{codim 2} to $Y = \cup_{i \neq j}  \left(D_i \cap D_j \right)$, for every $\epsilon >0$, there exists a complex algebraic variety $T$, a morphism $\psi \colon T \to S$ and a relative effective Cartier divisor $W$ in $X_T \slash T$ such that, for every $s\in S$ and every orbit $f \colon \bC \to X_s$ of a one-parameter subgroup of $A_s$ such that $f(\bC) \not \subset Y$, one has either 
\begin{enumerate}
    \item $  N^{(1)}(r, f, Y_s) \leq \epsilon \cdot T(r,f,L_s) || $, or
    \item $f( \bC)$ is contained in $W_t$ for some $t \in \psi^{-1}(s)$. (Here, we identify $X_t$ with $X_s$ using $\psi$.)
\end{enumerate}

Thanks to Proposition \ref{key induction step}, for every $\epsilon > 0$, there exists $\Xi \subsetneq X$ a closed algebraic subvariety such that, for every $s\in S$ and every orbit $f \colon \bC \to X_s$ of a one-parameter subgroup of $A_s$ such that $f( \bC)\subsetneq \Xi$ and $f( \bC)$ is contained \textit{up to translation} in $D_s$ or in $W_t$ for some $t \in \phi^{-1}(s)$ (where $X_t$ is identified with $X_s$ using $\phi$), the estimate
\[ T(r, f, D_s) \leq N^{(1)}(r, f, D_s) +  \epsilon \cdot T(r, f, L_s) \, || \,  \]
holds. 

If we are in the first case of the alternative, by applying Theorem \ref{truncated SMT for abelian schemes} to every $D_i$, we get that there exists a closed subscheme $ \Xi^\prime \subsetneq X$ such that for every orbit $f \colon \bC \to X_s$ of a one-parameter subgroup of $A_s$ whose image is not contained in $\Xi^\prime$, the following estimate holds:
\[ \sum_{i=1}^k T\left(r, f, \left(D_i\right)_s \right) \leq  \sum_{i=1}^k N^{(1)}\left(r, f, (D_i)_s \right) + k \cdot \epsilon \cdot T(r, f, L_s) \, || \, . \]
Therefore, in both cases, the estimate
\[ T(r, f, D_s) \leq N^{(1)}(r, f, D_s) + (k+1) \cdot \epsilon \cdot T(r, f, L_s) \, || \,  \]
holds for every orbit $f \colon \bC \to X_s$ of a one-parameter subgroup of $A_s$ whose image is not contained in $\Xi \cup \Xi^\prime$.
\end{proof}

\begin{claim} 
One can assume that the generic fibre of $D$ is geometrically integral. A fortiori, $D_s$ is integral for every $s \in S$.
\end{claim}
\begin{proof}
It is a consequence of Proposition \ref{etale cover} and Lemma \ref{reduce to geometrically integral}.
\end{proof} 

\begin{claim}
One can assume that the geometric generic fibre of $D$ is of general type. Therefore, up to shrinking $S$, one can assume that $D_s$ is of general type for every $s \in S$.
\end{claim}
\begin{proof}
If the geometric generic fibre of $D$ is not of general type, then by Theorem \ref{relative Ueno fibration}, up to shrinking $S$, there exists a non-trivial abelian $S$-subscheme $B \subset A$ and an integral effective Weil divisor $D^\prime \subset X \slash B$ such that $D = \pi^{-1} \left(D^\prime \right)$ where $\pi \colon X \to X \slash B$ is the quotient morphism. The divisor $D^\prime $ is a priori only a Weil divisor, but it can be made Cartier by shrinking $S$.\\
Choose a line bundle $M \to X \slash B$ which is relatively ample with respect to $S$. Then, up to shrinking $S$, there exists a positive integer $m$ and an effective Cartier divisor $E \subset X$ such that $L^{\otimes m} = \pi^\ast M  \otimes \cO_{X}(E)$. In particular, for every orbit $f \colon \bC \to X_s$ of a one-parameter subgroup of $A_s$ whose image is not contained in $E$, one has 
\[ T(r, \pi \circ f,  M_s) = T(r,f, \pi^\ast M_s) \leq m \cdot T(r,f, L_s) \]
thanks to Theorem \ref{Nevanlinna's inequality}.\\
Since the relative dimension of $X \to X \slash B$ and $ X \slash B \to S$ are both smaller than $d$, by induction one can apply Theorem \ref{truncated SMT for abelian schemes} to the tuples $(X \slash B, B_{X \slash B}, X, D, L)$ and $(S, A \slash B, X \slash B, D^\prime, M)$. Fix $\epsilon > 0$ and set $\Xi = \Xi(A, X, D, L, \epsilon)$ with
\[ \Xi := D \cup \Xi(X \slash B, B_{X \slash B}, X, D, L, \epsilon) \cup \pi^{-1} \left( \Xi(S, A \slash B, X \slash B, D^\prime, M, \frac{\epsilon}{m}) \right).\]
Let $f \colon \bC \to X_s$ be an orbit of a one-parameter subgroup of $A_s$ such that $f (\bC) \not \subset \Xi$. The composite morphism $\pi \circ f \colon \bC \to (X \slash B)_s$ is either constant or an orbit of a one-parameter subgroup of $(A \slash B)_s$. In the first case, the image of $f$ is contained in a fibre $X_t$ of $X \to X \slash B$ for some $t \in X \slash B$, and $f$ can be seen as an orbit $f \colon \bC \to X_t$ of a one-parameter subgroup of $B_t$ to which one can apply Theorem \ref{truncated SMT for abelian schemes} to get the wanted estimate
\[ T(r,f,D_s) \leq N^{(1)}(r,f, D_s) + \epsilon \cdot T(r,f,L_s) \; ||_{\epsilon} . \]
In the second case, by applying Theorem \ref{truncated SMT for abelian schemes} to $\pi \circ f \colon \bC \to \left(X \slash B \right)_s$ and observing that $N^{(1)}(r,f, D_s) =  N^{(1)}(r,\pi \circ f, D^\prime_s)$ since $f (\bC) \not \subset D$, we obtain also that
\[ T(r,f,D_s) \leq N^{(1)}(r,f, D_s) + \epsilon \cdot T(r,f,L_s) \; ||_{\epsilon} . \]
\end{proof}

We are left to consider the tuples $\left(S, A, X, D, L \right)$, where in addition $D_s \subset X_s$ is an integral effective Cartier divisor of general type for every $s \in S$. In particular, the following proposition implies that $\codim \left(  J_1(D/S)_t , J_1(X/S)_t  \right) \geq 2$ for every $t \in \Lie (A/S )$.

\begin{prop}
Let $A$ be a complex abelian variety and $W \subset A$ be an integral closed subscheme such that $\cap_{w \in W} T_w W \neq \{ 0 \}$ (for every $w \in W$, we use the canonical isomorphism $T_w A = \Lie(A)$ to see every $T_w W$ as a linear subspace of $\Lie(A)$). Then $W$ is not of general type.
\end{prop}
\begin{proof}
By assumption, the closed subscheme $W \subset A$ is stabilized under the action of a one-parameter subgroup $\bC \to A$, therefore it is stabilized under the Zariski closure of $\bC \to A$, which is a non-trivial abelian subvariety $B$ of $A$ by the Bloch-Ochiai Theorem \cite[Theorem 3]{Kawamata-Bloch}. It follows that the composite map $W \to A \to A \slash B$ is a fibration whose fibers are all isomorphic to $B$, hence $W$ cannot be of general type.
\end{proof}

Thanks to Theorem \ref{SMT for abelian schemes}, there exists an integer $\rho \geq 1$ such that  for every $s\in S$ and every orbit $f \colon \bC \to X_s$ of a one-parameter subgroup $\bC \to A_s$, one has either that
\begin{enumerate}
    \item the image of a translate of $f$ is contained in $D_s$, or
    \item $T (r, f, D_s) \leq  N^{(\rho)}(r, f, D_s) +o\left(T \left(r,f,L_s \right)\right) ||$.
\end{enumerate}

The first case is handled by Proposition \ref{key induction step}, therefore it is sufficient to consider the second case.  If $\rho = 1$ we are done, so that we can assume that $\rho \geq 2$. Observing that one has
\[  N^{(2)}\left(r, f, D_s \right) -  N^{(1)}\left(r, f, D_s \right) \leq N^{(1)} \left(r, J_1(f), J_1(D / S)_s \right), \]
and 
\[ N^{(\rho)}(r, f, D_s) - N^{(1)}(r, f, D_s) \leq (\rho - 1) \cdot \left( N^{(2)}(r, f, D_s) - N^{(1)}(r, f, D_s) \right), \]
it follows that
\[ N^{(\rho)}(r, f, D_s) \leq N^{(1)}(r, f, D_s) + (\rho - 1) \cdot N^{(1)}(r, J_1(f), J_1 (D / S)_s). \]

Therefore, it remains to bound the function $r \mapsto N^{(1)}(r, J_1(f), J_1 (D / S)_s)$. Note that $J_1(f) \colon \bC \to J_1(X/ S)_s$ is an orbit of a one-parameter subgroup in $J_1(X / S) = X \times_S \Lie (A /S)$, where the later is seen as a torsor over $\Lie (A /S)$ under the abelian $\Lie (A /S)$-scheme $J_1(A / S) = A \times_S \Lie (A /S)$. Let $M \to J_1(X / S)$ denote the pull-back of $L \to X$ along the morphism $J_1(X / S) \to X$. Since $J_1 (D/ S)_t$ has codimension at least $2$ in $J_1 (X/S)_t$ for every $t \in \Lie (A /S)$, one can apply Theorem \ref{codim 2} to the tuple $\left(\Lie (A/S), J_1 (A/S) \to \Lie (A/S), J_1 (X/S) \to \Lie (A/S), M, J_1 (D/S) \right)$.

Therefore, for every $\epsilon >0$, there exists a complex algebraic variety $U$, a morphism $\psi \colon U \to \Lie (A/S)$ and a relative effective Cartier divisor $W$ in $X_U \slash U$ such that, for every $t \in \Lie (A/S)$ and every orbit $g \colon \bC \to J_1(X/S)_t$ of a one-parameter subgroup $\bC \to J_1(A/S)_t$ such that $g(\bC) \not \subset J_1 (D/S)$, one has either 
\begin{enumerate}
    \item  $  N^{(1)}(r,  g, J_1(D/S)) \leq \frac{\epsilon}{2(\rho - 1)} \cdot T(r,g,M_t) || $, or
    \item $g( \bC)$ is contained in $W_u$ for some $u \in \psi^{-1}(t)$. (Here, we identify $X_u$ with $X_u$ using $\psi$.)
\end{enumerate}

By applying this to $g = J_1(f)$, the first possibility becomes  
\[  N^{(1)}(r,  J_1(f) , J_1(D/S)) \leq \frac{\epsilon}{2(\rho - 1)} \cdot T(r,f, L_s) || \]
 thanks to (\ref{relating T for f and J_1(f)}), so that we get the estimate
 \[T (r, f, D_s) \leq  N^{(1)}(r, f, D_s) +  \frac{\epsilon}{2} \cdot T(r,f,L_s)  +o\left(T(r,f,L_s)\right) ||,   \]
and the result follows. On the other hand, thanks to Proposition \ref{key induction step}, the second possibility is taken care by induction.


\section{Proofs of Theorem \ref{relative Lang conjecture in the abelian case} and Theorem \ref{main corollary}}

\subsection{A reduction}
Let $\cP$ be a property of proper morphisms of complex algebraic varieties. Assume that the following properties are satisfied:
\begin{enumerate}
	\item For every complex algebraic variety $X$, the identity morphism $X \to X$ satisfies $\cP$,
    \item Let $f \colon X \to Y$ be a proper morphism of complex algebraic varieties satisfying $\cP$. If $g \colon Z \to X$ is a finite morphism of complex algebraic varieties, then the composite morphism $f \circ g \colon Z \to Y$ satisfies $\cP$.
    \item Let $f \colon X \to Y$ be a proper morphism of complex algebraic varieties satisfying $\cP$. If $Z \rightarrow Y$ is a  morphism of complex algebraic varieties, then the induced proper morphism $f_{Z} \colon X_Z = X \times_Y Z \to Z$ satisfies $\cP$.
\end{enumerate}
\begin{prop}\label{property P-Stein factorization}
Let $X \to Y$ be a proper morphism with Stein factorization $X \to Z \to Y$. If $X \to Y$ satisfies $\cP$, then $X \to Z$ satisfies $\cP$. 
\end{prop}
\begin{proof}
The morphism $X \to Z$ can be written as the composition of the closed immersion $X \subset X \times_Y Z$ with the morphism $ X \times_Y Z \to Z$ obtained by base-change from the morphism $X \to Y$.
\end{proof}

\begin{prop}\label{first reduction}
Fix $\ast \in \{h,ab, alg \}$. Assume that the following holds:
\begin{enumerate}
    \item Let $X$ be a proper integral complex algebraic variety such that the structural morphism $X \to \mathrm{Spec}(\bC)$ satisfies $\cP$. If the complex algebraic variety $X$ is not of general type, then $\Sp_\ast(X) = X$.
    \item Let $f \colon X \to Y$ be a surjective proper morphism between two integral complex algebraic varieties. Assume that $f$ satisfies $\cP$ and that its generic fibre is geometrically integral. If $f$ is of general type, then $\Sp_\ast(X \slash Y)$ is not Zariski dense in $X$.
\end{enumerate}
Then, for every proper morphism $X \to Y$ between complex algebraic varieties satisfying $\cP$, the following holds:
\begin{enumerate}
\item $\Sp_\ast(X \slash Y)$ is a Zariski closed algebraic subset of $X$,
\item $\Sp_\ast(X \slash Y) \neq X$ if and only if $X \to Y$ is of general type.
\end{enumerate}  
\end{prop}

\begin{proof}
Let $f \colon X \to Y$ be a proper morphism between two complex algebraic varieties. Assume that $f$ satisfies $\cP$. Up to replacing $Y$ with the image of $X$ in $Y$, one can assume that $X \to Y$ is surjective.\\

Assume that $f$ is not of general type, and let us prove that $\Sp_\ast\left(X \slash Y\right) = X$. If $Z \subset X$ is an irreducible component of $\left(X_y \right)^{\mathrm{red}}$ for some $y \in Y$, then $Z$ is not of general type thanks to Theorem \ref{generization general type}. Observe that $Z \to \Spec(\bC)$ satisfies $\cP$, since $\cP$ is stable by base-change and closed immersion. It follows from the assumptions that $\Sp_\ast\left(Z \slash Y\right) = Z$. This proves that $\Sp_\ast \left(X \slash Y \right) = X$.\\

Assume that $f$ is of general type, and let us prove that $\Sp_\ast \left(X \slash Y \right)$ is not Zariski-dense in $X$. Thanks to Proposition \ref{special sets-union}, one can assume that $X$ is integral. Thanks to Proposition \ref{normalization-general type} and Proposition \ref{special sets-normalization}, one can assume moreover that $X$ is normal. Let $X \to Z \to Y$ be the Stein factorization of the proper morphism $X \to Y$. Thanks to Proposition \ref{property P-Stein factorization}, since $X \to Y$ satisfies $\cP$, $X \to Z$ satisfies $\cP$. Moreover, thanks to Proposition \ref{Stein factorization-general type}, the fibration $X \to Z$ is of general type. Since $X$ is normal and integral, $Z$ is normal and integral by Proposition \ref{fibration-normal and integral}, and the generic fibre of $X \to Z$ is geometrically integral by Proposition \ref{fibration-generic fibre}. Using the assumptions, it follows that $\Sp_\ast(X \slash Z)$ is not Zariski dense in $X$. Since, $\Sp_\ast(X \slash Y) = \Sp_\ast(X \slash Z)$, we get that $\Sp_\ast(X \slash Y)$ is not Zariski dense in $X$.\\

Finally, let us prove that $\Sp_\ast(X \slash Y)$ is a Zariski closed subset of $X$. Assume first that $X \to Y$ is not of general type. Then we have seen that $\Sp_\ast(X \slash Y) = X$, so that $\Sp_\ast(X \slash Y)$ is clearly Zariski closed in $X$. Assume now that $X \to Y$ is of general type. Then we know that the set $S :=\Sp_\ast(X \slash Y)$ is not Zariski-dense in $X$. Let $\bar S$ be its Zariski closure in $X$. Clearly $\Sp_\ast(\bar S \slash Y) = \Sp_\ast(X \slash Y) = S$, hence $\Sp_\ast(\bar S \slash Y)$ is evidently Zariski dense in $\bar S$. It follows that $\bar S \to Y$ is not of general type, so that $\Sp_\ast(\bar S \slash Y) = \bar S$ by the preceding discussion. This proves that $S = \bar S$ is Zariski closed in $X$.
\end{proof}

\subsection{The surjective case}

\begin{thm}\label{key result-finite surjective}
Let $S$ be an integral complex algebraic variety. Let $A \to S$ be an abelian scheme. Let $Y \to S$ be an $A$-torsor over $S$. Let $\pi \colon X \to Y$ be a finite surjective morphism. Assume that the generic fibre of the composite morphism $X \to S$ is geometrically integral. If $X \to S$ is of general type, then $\Sp_h(X / S)$ is not Zariski dense in $X$.
\end{thm}

Before proving Theorem \ref{key result-finite surjective}, we need a first result.

\begin{lem}\label{twisted ramification formula}
Let $X \to S$ and $Y \to S$ be two smooth fibrations between complex integral schemes. Let $\pi \colon X \to Y$ be a generically finite surjective morphism. Then, up to shrinking $S$, there exists a reduced relative effective Cartier divisor $G \subset Y$ and a relative effective Cartier divisor $H \subset X$ such that, letting $D := \left( \pi^{-1} G \right)^{red}$, 
\[ \pi^\ast \left( \omega_{Y / S}\left( G \right) \right) = \omega_{X / S} \left(D + H \right) . \]
\end{lem}
\begin{proof} When $S$ is a field, this is proved in \cite[Proof of Corollary 3.23]{Yamanoi-lectures}. The general case follows by applying the lemma over the generic point $\eta$ of $S$ and observing that everything extends over a non-empty Zariski-open neighborhood of $\eta$ in $S$.
\end{proof}

\begin{proof}[Proof of Theorem \ref{key result-finite surjective}]
We keep the notations from the statement. Observe that one can freely replace $S$ by a dense Zariski open subset. More generally, thanks to Corollary \ref{base change-special sets}, one can replace $S$ by an étale cover, therefore one can assume from the beginning that $Y \to S$ admits a section, that one can use to identify $Y$ with $A$.\\

1. Up to shrinking $S$, there exists a smooth $S$-variety $\tilde{X}$ and  a proper birational morphism $\nu \colon \tilde{X} \to X$. Write $\tilde{\pi} = \pi \circ \nu \colon \tilde{X} \to A$. Fix a line bundle $L \to A$ which is relatively ample over $S$ and $M \to \tilde {X}$ a line bundle which is relatively ample over $S$.\\

2. Let $Z \subsetneq X$ be the smallest closed subscheme such that $\nu$ is an isomorphism over $X \backslash Z$. In particular, every non-constant holomorphic map $f \colon \bC \to X$ whose image is not contained in $Z$ lifts uniquely to a holomorphic map $\tilde{f} \colon \bC \to \tilde{X}$.\\

3. Thanks to Lemma \ref{twisted ramification formula}, up to shrinking $S$, there exists a reduced relative effective Cartier divisor $G \subset A$ and a relative effective Cartier divisor $H \subset \tilde{X}$ such that, letting $D := \left( \tilde{\pi}^{-1} G \right)^{red}$, 
\begin{equation*}
\tilde{\pi}^\ast \left( \cO_A \left( G  \right)\right) =  \tilde{\pi}^\ast \left( \omega_{A / S}\left( G \right) \right) = \omega_{\tilde{X} / S} \left(D + H \right) .     
\end{equation*}
(Note that $\omega_{A / S}$ is trivial since $A / S$ is an abelian scheme.)\\

4. Since the morphism $\tilde{X} \to S$ is of general type, by applying Kodaira's lemma to the generic fibre of $\tilde{X} \slash S$, it follows that, up to shrinking $S$, there exists $l$ a positive integer and $E \subset \tilde{X}$ an effective Cartier divisor such that 
\[ \left( \omega_{\tilde{X} / S} \left(H \right) \right)^{\otimes l} = \tilde{\pi}^\ast L \otimes M \otimes \cO_{\tilde{X}}(E). \]  
In particular, for every $s \in S$ and every non-constant holomorphic map $\tilde{f} \colon \bC \to \tilde{X}_s$ such that $\tilde{f}(\bC) \not \subset E \cup H$, we have the following inequality
\begin{equation}\label{from ample to big}
   T(r, \tilde{f}, \omega_{\tilde{X} / S} \left(H \right)_s ) \geq  \frac{1}{l} \cdot \left( T(r, \tilde{f}, M_s) +  T(r,\pi \circ f,L_s) \right)  \; || .  
\end{equation} 

5. Applying Theorem \ref{truncated SMT for abelian schemes} to the reduced effective Cartier divisor $G \subset A$ and $\epsilon = \frac{1}{l}$, there exists $\Xi = \Xi(A, G, L, \epsilon)  \subsetneq A$ a closed subscheme such that for every $s \in S$, if $f \colon \bC \to A_s$ is a one-parameter subgroup such that $f (\bC) \not \subset \Xi$, then $G_s$ is a reduced effective Cartier divisor and 
\begin{equation}\label{SMT}
T(r,f,G_s) \leq N^{(1)}(r,f, G_s) + \epsilon \cdot T(r,f, L_s) \; ||_{\epsilon} .     
\end{equation}

6. To finish the proof, we will check that $\Sp_h(X / S) \subset \pi^{-1}(\Xi) \cup \nu(E \cup H) \cup Z$. Fix $s \in S$ and let $f \colon \bC \to X_s$ be a non-constant holomorphic map. By a theorem of Kawamata \cite[Theorem 3]{Kawamata-Bloch}, the Zariski-closure of $f(\bC)$ in $X_s$ is an abelian variety $B \subset X_s$. Moreover, letting $C := \pi(B)$, the induced morphism $B \to C$ is finite étale. It follows that $\Sp_h(X_s)$ is covered by the images of those holomorphic maps $f\colon \bC \to X_s$ such that the composition $\pi \circ f \colon \bC \to A_s$ is a one-parameter subgroup.\\

7. Fix $s \in S$. Let $f \colon \bC \to X_s$ be a holomorphic map and assume that the composition $\pi \circ f \colon \bC  \to A_s$ is a one-parameter subgroup. Assume also that $f(\bC) \not \subset Z$, so that $f$ lifts uniquely to a holomorphic map $\tilde{f} \colon \bC \to \tilde{X}_s$. If moreover $\pi \circ f (\bC) \not \subset \Xi$, then, by $\eqref{SMT}$, $G_s$ is a reduced effective Cartier divisor and 
\begin{equation*}
T(r,\pi \circ f,G_s) \leq N^{(1)}(r,\pi \circ f, G_s) + \epsilon \cdot T(r,\pi \circ f,L_s) \; ||_{\epsilon} .     
\end{equation*}

Since $ \tilde{\pi}^\ast \left( \cO_A \left( G  \right)\right)  = \omega_{\tilde{X} / S} \left(D + H \right) $, we have 
\begin{equation*}
T \left(r, \tilde{f}, \left( \omega_{\tilde{X} / S} \left(D + H \right) \right)_s \right) =
T\left(r, \pi \circ f, G_s \right),    
\end{equation*}
whereas the equality $D = \left( \tilde{\pi}^{-1} G \right)^{red}$ implies that
\begin{equation*}
N^{(1)}(r,\pi \circ f, G_s) = N^{(1)}(r,\tilde{f}, D_s).    
\end{equation*}

It follows from the last three equations that
\begin{equation*}
T \left(r, \tilde{f}, \omega_{\tilde{X} / S} \left(D + H \right)_s \right) \leq N^{(1)}(r,\tilde{f}, D_s)  + \epsilon \cdot T(r,\pi \circ f,L_s) \; ||_{\epsilon} .    
\end{equation*}

Using Nevanlinna's inequality (cf. Theorem \ref{Nevanlinna's inequality})
\[ N^{(1)}(r,\tilde{f}, D_s) \leq  T(r,\tilde{f}, D_s),\]
we obtain that
\begin{equation*}
T(r, \tilde{f}, \omega_{\tilde{X} / S} \left(H \right)_s ) \leq   \epsilon \cdot T(r,\pi \circ f,L_s)  \; ||_{\epsilon} .     
\end{equation*}
This implies that $\tilde{f}(\bC) \subset E \cup H$, since otherwise, recalling that $\epsilon = \frac{1}{l}$, the equation \eqref{from ample to big} would imply
\begin{equation*}
T(r, \tilde{f}, M_s) \leq   0  \; ||_{\epsilon} ,    
\end{equation*}
a contradiction since $\tilde{f}$ is not constant and $M_s$ is ample. 
\end{proof}

\subsection{End of the proof of Theorem \ref{relative Lang conjecture in the abelian case}}

From now on, say that a proper morphism $X \to Y$ between two complex algebraic varieties satisfies $\cP$ if there exists an abelian scheme $A \to Y$, an $A$-torsor $P \to Y$ and a finite $Y$-morphism $X \to P$. Then $\cP$ satisfies evidently the stability properties stated at the beginning of this section. Recall the following result of Kawamata.

\begin{thm}[Kawamata, {\cite[Theorem 13]{Kawamata-characterization}}] 
Let $X$ be a normal proper integral complex algebraic variety which admits a finite morphism $X \to A$ to an abelian variety $A$. Up to replacing $X$ with a finite étale cover, $X$ is isomorphic to a product $B \times X^\prime$ of an abelian variety $B$ and a normal proper integral complex algebraic variety $X^\prime$ whose dimension is equal to the Kodaira dimension $\kappa(X)$ of $X$.
\end{thm}

In particular, if a proper integral complex algebraic variety $X$ admits a finite morphism to an abelian variety, and if $X$ is not of general type, then $\Sp_{alg}(X) = \Sp_{ab}(X) = \Sp_h(X) = X$.  Therefore, thanks to Proposition \ref{first reduction}, in order to prove Theorem \ref{relative Lang conjecture in the abelian case}, one is reduced to prove the following special case:

\begin{thm}\label{key result}
Let $X \to S$ be a proper morphism between two integral complex algebraic varieties. Assume that $S$ is integral and that the generic fibre of $X \to S$ is geometrically integral. Assume that there exists an abelian scheme $A \to S$, an $A$-torsor $P \to S$ and a finite $S$-morphism $X \to P$. If $X \to S$ is of general type, then $\Sp_h(X / S)$ is not Zariski-dense in $X$.
\end{thm}

To prove Theorem \ref{key result}, one can also freely replace $S$ by any non-empty Zariski-open subset. We already know that Theorem \ref{key result} holds when $X \to P$ is either a closed immersion or finite surjective, cf. Theorem \ref{key result-closed immersion} and Theorem \ref{key result-finite surjective}. The general case can be reduced to these two special cases as follows.\\

We are in position to apply Corollary \ref{corollary Ueno fibration}, so that up to shrinking $S$, there exists a diagram

\[
\begin{tikzcd}
X \arrow[rd] \arrow[r] & W \arrow[d] \arrow[r, hookrightarrow] & P \arrow[d] \\
  & Z \arrow[r, hookrightarrow]& P \slash B
\end{tikzcd}
\]
such that 
\begin{enumerate}
\item $Z \to S$ is a proper morphism of general type,
\item the generic fibre of $Z \to S$ is geometrically integral,
\item $P / B$ is an $A/B$-torsor,
\item the morphism $X \to W$ is finite surjective,
\item and $W \to Z$ is a torsor under the abelian scheme $B_Z$.
\end{enumerate}

Since $Z \to S$ is a proper morphism of general type with a geometrically integral generic fibre and $Z \subset P \slash B$ is a closed immersion, it follows from Theorem \ref{key result-closed immersion} that $\Sp_h(Z \slash S)$ is not Zariski-dense in $Z$. On the other hand, thanks to Proposition \ref{composition general type}, the proper surjective morphism $X \to Z$ is of general type, because the composite morphism $X \to Z \to S$ is a proper surjective morphism of general type. Since the morphism $X \to W$ is finite surjective and $W \to Z$ is a torsor under an abelian scheme, it follows from Theorem \ref{key result-finite surjective} that $\Sp_h(X \slash Z)$ is not Zariski-dense in $X$. Thanks to Proposition \ref{composition-special sets}, we conclude that $\Sp_h(X \slash S)$ is not Zariski-dense in $X$.

This concludes the proofs of Theorem \ref{key result} and Theorem \ref{relative Lang conjecture in the abelian case}.

\subsection{Proof of Theorem \ref{main corollary}}

Let $f\colon X \to Y$ be a smooth proper surjective morphism with connected fibres between complex algebraic varieties. Assume that there exists a geometric point $\bar \eta \colon \Spec \left( \Omega \right) \to Y$ such that $X_{\bar \eta}$ is of general type and of maximal Albanese dimension. One can assume without any loss of generality that $Y$ is integral. Up to replacing $Y$ with a desingularization, one can assume that $Y$ is non-singular thanks to Corollary \ref{base change-special sets}. Moreover, up to replacing $Y$ with an étale cover whose image contains the image of $\bar \eta$, one can assume that $f \colon X \to Y$ admits a section $\sigma \colon Y \to X$. Let $\Alb(X/Y) \to Y$ denote the relative Albanese morphism, and let $X \to \Alb(X/Y)$ be the $Y$-morphism associated to the section $\sigma$. Let $X \to Z \to \Alb(X/Y)$ be the Stein factorization of the proper morphism $X \to \Alb(X/Y)$. Since by assumption $X_{\bar \eta}$ is of maximal Albanese dimension, one has that $\dim(X_{\bar \eta})= \dim(Z_{\bar \eta})$. But the induced morphism $X_{\bar \eta} \to Z_{\bar \eta}$ is a fibration, hence it has to be an isomorphism over a non-empty open of $Z_{\bar \eta}$ (recall that for any proper morphism of schemes $X \to Z$, the function which associates to $z \in Z$ the dimension of $X_z$ is upper semi-continuous). In particular, the fibration $g \colon X \to Z$ admits a geometric fibre which is isomorphic to a (reduced) point, hence $g$ is birational, i.e. there exists a dense Zariski open subset $U \subset Z$ such that the induced morphism $g_{|g^{-1}(U)} \colon g^{-1}(U) \to U$ is an isomorphism. If $\bar \mu \colon \Spec\left( \Omega \right) \to Y$ is a geometric point lying over the generic point of $Y$, this implies that the induced $\Omega$-morphism $X_{\bar \mu} \to Z_{\bar \mu}$ is birational. Thanks to Theorem \ref{generization general type}, the $\Omega$-variety $X_{\bar \mu}$ is of general type, hence the (normal proper integral) $\Omega$-variety $Z_{\bar \mu}$ is of general type. In other words, the proper morphism $Z \to Y$ is of general type. Since the morphism $Z \to \Alb(X/Y)$ is finite by definition, it follows from Theorem \ref{relative Lang conjecture in the abelian case} that, for every $\ast \in \{alg, ab , h\}$, $\Sp_\ast \left(Z / Y \right)$ is not Zariski dense in $Z$. Thanks to Proposition \ref{composition-special sets}, it follows that $\Sp_\ast \left(X/ Y \right)$ is not Zariski dense in $X$.\\

To conclude, let us remark that the proof above gives the following weak analogue of Theorem \ref{generization general type}.

\begin{prop}
Let $X \to Y$ be a surjective smooth proper morphism  between irreducible complex varieties. If there exists a geometric point $\bar \eta \colon \Spec \left( \Omega \right) \to Y$ such that $X_{\bar \eta}$\ is of maximal Albanese dimension, then, for every geometric point $\bar \mu \colon \Spec \left( \Omega \right) \to Y$ lying over the generic point of $Y$, the smooth proper $\Omega$-variety $X_{\bar \mu}$\ is of maximal Albanese dimension.
\end{prop}

\small

\bibliography{./biblio}

\bibliographystyle{amsalpha}

\end{document}